\documentclass[a4paper,oneside,10pt,final,reqno]{amsart}

\usepackage[arrow,line,matrix]{xy}
\usepackage{showkeys}

\newcommand{\bb} {\mathbb}
\newcommand{\cal}{\mathcal}
\newcommand{\frk}{\mathfrak}

\newcommand{\ep}{\epsilon}
\newcommand{\ve}{\varepsilon}

\newcommand{\vac}[1]{\left|\, #1\, \right>}

\newcommand{\loc}{\text{loc}}

\newcommand{\ch}{\operatorname{ch}}

\newcommand{\Coh}{\operatorname{Coh}}

\newcommand{\Hom}{\operatorname{Hom}}

\newcommand{\Res}{\mathop{\operatorname{Res}}}
\newcommand{\Vir}{\mathcal{V}ir}

\newcommand{\dd}{\oalign{\raisebox{0.4em}[0.5em][0em]{$\circ$} \cr
                         \raisebox{0.15em}[0.5em][0em]{$\circ$}}}
\newcommand{\nord}[1]{\dd #1 \dd}

\theoremstyle{plain}
 \newtheorem{thm}{Theorem}[section]
 \newtheorem{lem}[thm]{Lemma}
 \newtheorem{prop}[thm]{Proposition}
 \newtheorem{cor}[thm]{Corollary}
 \newtheorem*{thm*}{Theorem}
\theoremstyle{definition}
 \newtheorem{dfn}[thm]{Definition}
 
 \newtheorem{fct}[thm]{Fact}
 \newtheorem{rmk}[thm]{Remark}
\theoremstyle{remark}

\numberwithin{equation}{section}
\allowdisplaybreaks

\begin{document}


\title{Whittaker vector of deformed Virasoro algebra and Macdonald symmetric functions}

\author{Shintarou Yanagida}
\address{Research Institute for Mathematical Sciences,
Kyoto University, Kyoto 606-8502, Japan}
\email{yanagida@kurims.kyoto-u.ac.jp}

\date{February 12, 2014}


\begin{abstract}
We give a proof of Awata and Yamada's conjecture for 
the explicit formula of Whittaker vector 
of the deformed Virasoro algebra realized in the Fock space.
The formula is expressed as a summation over 
Macdonald symmetric functions with factored coefficients.
In the proof we fully use currents appearing in the Fock representation 
of Ding-Iohara-Miki quantum algebra.
We also mention an interpretation of Whittaker vector 
in terms of the geometry of the Hilbert schemes of points on the affine plane.
\end{abstract}

\maketitle
\tableofcontents

\section{Introduction}

In \cite{AY} 
Awata and Yamada studied an analog of 
the Alday-Gaiotto-Tachikawa relation \cite{AGT}
in five dimensions. 
They conjectured that Nekrasov's instanton partition function of 
$5D$ $N = 1$ pure $\mathrm{SU}(2)$ gauge theory 
coincides with the inner product of the Gaiotto-like state \cite{G}
in the deformed Virasoro algebra.

Let us briefly explain the notions appearing here.
The deformed Virasoro algebra $\Vir_{q,t}$ was introduced in \cite{SKAO} 
as the (topological) algebra generated by the current 
$T(z)=\sum_{n\in\bb{Z}}T_n z^{-n}$ satisfying the relation
$$
 f(w/z) T(z) T(w) -T(w)T(z) f(z/w) = 
 -\dfrac{(1-q)(1-t^{-1})}{1-q/t}
 \Bigl(\delta( q t^{-1} w/z) - \delta(t q^{-1} w/z)\Bigr)
$$
with $f(x) := \exp(\sum_{n\ge1}(1-q^n)(1-t^{-n})/(1+q^n/t^n)\cdot x^n/n)$.

The algebra $\Vir_{q,t}$ has Verma module, an analogue of Verma modules of 
triangular decomposed Lie algebras.
It is $\bb{Z}_{\ge0}$-graded and generated by a highest weight vector.
The Verma module with highest weight $h$ is denoted by $M_h$
and its $n$-th grading part is denoted by $M_{h,n}$.
The completion of $M_h$ with respect to the $\bb{Z}_{\ge0}$-grading is denoted by $\widehat{M}_h$.

Let $L$ be an element of $\bb{F}$.
The Gaiotto-like state in $\Vir_{q,t}$
is an element $v_G$ of the completed Verma module $\widehat{M}_h$ 
with the property
\begin{align}\label{eq:intro:whit}
 T_1 v_G = L v_G, \qquad 
 T_n v_G = 0 \quad (n \ge 2).
\end{align}
One immediately finds that $v_G$ is of the form 
$v_G = \sum_{n\ge0}L^n v_{G,n}$ (if it exists), 
where $v_{G,n} \in M_{h,n}$ is independent of $L$.
Thus the choice of $L$ is inessential.

The Gaiotto-like state turned out to be an analogue of the Whittaker vector 
for finite dimensional Lie algebra \cite{K}
as the condition \eqref{eq:intro:whit} implies.
So let us call \eqref{eq:intro:whit} the Whittaker condition,
and call the vector $v_G$ the Whittaker vector for $\Vir_{q,t}$.

Along the way of checking the five dimensional relation,
Awata and Yamada also discovered an explicit formula \cite[(3.18)]{AY}
of $v_G$ realized in the Fock representation.
The formula is expressed in a form of summation over 
Macdonald symmetric functions with factored coefficients.

\begin{thm*}
The Whittaker vector $v_G$ of the deformed Virasoro algebra exists uniquely,
and in the Fock representation, 
its $n$-th degree part is expressed by 
\begin{align*}
v_{G,n} =
  \sum_{\lambda \, \vdash n} P_{\lambda}(q,t) \gamma_\lambda,
\quad
  \gamma_\lambda =
  \prod_{s \in \lambda}
  \dfrac{k}{1 - k^2 q^{j(s)} t^{-i(s)}}
  \dfrac{q^{a_\lambda(s)}}{1 - q^{1 + a_\lambda(s)} t^{l_\lambda(s)}}
\end{align*}
\end{thm*}
Here the symbol $k$ is the highest weight of Fock space,
and is related to the highest weight of the Verma module 
by $h=k+k^{-1}$.

In this paper we give a proof of this explicit formula of Whittaker vector 
of the deformed Virasoro algebra realized in the Fock space.
Our proof will be done in the representation theory of quantum algebra.
The result can be viewed as a $q$-analogue of the explicit formula 
for Whittaker vector of the Virasoro Lie algebra,
which was proved in \cite{Y}.
However, the method of the proof is totally different.
In the non-deformed Virasoro case, the existence of the Whittaker vector 
follows from the universal property of the Verma module of Lie algebra.
For the deformed Virasoro algebra, no general theory like \cite{K} 
is not available, since it is not a Lie algebra, nor a vertex (operator) algebra. 

In the proof we rewrite the Whittaker condition \eqref{eq:intro:whit}
in the language of the Fock representation of the so-called 
Ding-Iohara-Miki quantum algebra $\cal{U}(q,t)$.
It is a version of (topological) Hopf algebra introduced by Ding and Iohara \cite{DI:1997}.
The algebra has a Fock representation, 
and this representation is intimately connected to the theory 
of Macdonald symmetric functions \cite{M}.
For a simple example, two parameters $q$ and $t$ in the algebra $\cal{U}(q,t)$
can be identified with the parameters in the Macdonald symmetric functions.
This surprising connection was discovered by Miki \cite{Mi}.
A few years later, several groups \cite{FHHSY,FT,FFJMM,SV} ``rediscovered" 
the algebra $\cal{U}(q,t)$ in several contexts.
See \S\ref{ssec:DIM} for more accounts.

By Feigin and Tsymabliuk's work \cite{FT} and 
Schiffmann and Vasserot's work \cite{SV},
the Fock representation of the Ding-Iohara-Miki algebra
is also related to the torus equivariant $K$-group 
of the Hilbert schemes of points on the affine plane.
Thus it is natural that our formula for Whittaker vector 
is related to the geometry of the Hilbert schemes of points.
After the proof of the formula, 
we will mention to this connection to the geometry.

Let us explain the organization of the paper here.
In \S \ref{sec:statement} we introduce the deformed Virasoro algebra, 
its Fock representation and symmetric functions.
After defining the Whittaker vector for Virasoro algebra,
we state in Theorem \ref{thm:main}
the conjecture of Awata and Yamada, whose proof is the main result of this paper.

\S \ref{sec:prelim} is a preliminary part of the proof.
After introducing Ding-Iohara-Miki algebra $\cal{U}_{q.t}$,
we rewrite the Whittaker condition of the deformed Virasoro algebra 
in terms of currents of operators in $\cal{U}_{q.t}$ acting on the Fock space.

In \S \ref{sec:prf} we give the proof.
Since the Whittaker condition can be written in three parts,
the proof is also separated in three parts.
The third condition is hardest to prove.

In \S \ref{sec:geom} we will give an interpretation of the Whittaker vector 
in terms of the geometry of Hilbert schemes of points on the plane.
In \S\ref{ssec:Haiman} we recall Haiman's work \cite{H:2001,H:2002,H:2003} 
and introduce the main morphism $\Phi$ from geometry to representation.
In \S\ref{ssec:geom:whit} we interpret the formula in terms of geometry.

The appendix is attached for some accounts of the computation 
in terms of the deformed Virasoro currents.
In the main text the computation will be done 
in terms of the currents $\eta(z)$ and $D(z)$.
Although we don't have fully general formulae for $T_d v_{G,n}$,
we can study a few properties of $T_d v_{G,n}$,
which seem to be interesting from combinatorial viewpoints.

\subsection*{Notation}

We follow \cite{M} for the notations of partitions and symmetric functions.

By a partition we mean a non-decreasing finite sequence of positive integers.
For a partition $\lambda$, its total sum is denoted by $|\lambda|$ 
and its length is denoted by $\ell(\lambda)$.
$\lambda \vdash n$ means that $\lambda$ is a partition with the condition 
$|\lambda| = n$.

A partition will be identified with the associated Young diagram.
We use $(i,j)$ for the coordinate of a box in a Young diagram, 
so that for a partition $\lambda=(\lambda_1,\lambda_2,\ldots,\lambda_\ell)$ 
we have $1 \le i \le \ell$ and $1 \le j \le \lambda_i$.

The arm and leg of the box $s$ at $(i,j)$ in the Young diagram $\lambda$ 
is denoted by $a_\lambda(s)$ and $l_\lambda(s)$.
We have $a_\lambda(s) = \lambda_i -j$ 
and $l_\lambda(s)= \lambda^\vee_j - i$.
Here $\lambda^\vee$ is the transposed Young diagram of $\lambda$.

In Figure \ref{fig:1} we give an example of Young diagram and 
combinatorial symbols.
We will use French notation for Young diagram.
The figure is the diagram for the partition $(9,9,6,5,2,2)$.
For the box $s$ shadowed in the figure,
we have $i(s) = 2$, $j(s)=4$, $a(s)= 5$ and $l(s)=2$.

\begin{figure}[htbp]
{\unitlength 0.1in%
\begin{picture}( 18.0000, 10.0000)(  2.0000,-12.0000)%
\special{pn 8}%
\special{pa 200 200}%
\special{pa 200 1200}%
\special{fp}%
\special{pa 400 200}%
\special{pa 400 1200}%
\special{fp}%
\special{pa 600 200}%
\special{pa 600 1200}%
\special{fp}%
\special{pa 800 400}%
\special{pa 800 1200}%
\special{fp}%
\special{pa 1000 400}%
\special{pa 1000 1200}%
\special{fp}%
\special{pa 1200 400}%
\special{pa 1200 1200}%
\special{fp}%
\special{pa 1400 600}%
\special{pa 1400 1200}%
\special{fp}%
\special{pa 1600 800}%
\special{pa 1600 1200}%
\special{fp}%
\special{pa 1800 800}%
\special{pa 1800 1200}%
\special{fp}%
\special{pa 2000 800}%
\special{pa 2000 1200}%
\special{fp}%
\special{pa 200 200}%
\special{pa 600 200}%
\special{fp}%
\special{pa 200 400}%
\special{pa 1200 400}%
\special{fp}%
\special{pa 200 600}%
\special{pa 1400 600}%
\special{fp}%
\special{pa 200 800}%
\special{pa 2000 800}%
\special{fp}%
\special{pa 200 1000}%
\special{pa 2000 1000}%
\special{fp}%
\special{pa 200 1200}%
\special{pa 2000 1200}%
\special{fp}%
\special{pn 4}%
\special{pa 990 810}%
\special{pa 810 990}%
\special{fp}%
\special{pa 1000 860}%
\special{pa 860 1000}%
\special{fp}%
\special{pa 1000 920}%
\special{pa 920 1000}%
\special{fp}%
\special{pa 940 800}%
\special{pa 800 940}%
\special{fp}%
\special{pa 880 800}%
\special{pa 800 880}%
\special{fp}%
\end{picture}}%

\caption{Young diagram for $(9,9,6,5,2,2)$.}
\label{fig:1}
\end{figure}
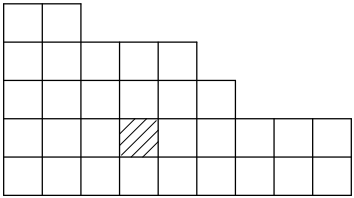

The space of symmetric functions with the coefficient ring $\bb{Z}$ 
will be denoted by $\Lambda$,
and we set $\Lambda_R := \Lambda \otimes_{\bb{Z}} R$ for a ring $R$.
The degree decomposition of $\Lambda$ will be denoted by 
$\Lambda = \oplus_{d=0}^{\infty} \Lambda_d$. 
The coefficient extension of $\Lambda_d$ is denoted by
$\Lambda_{R,d}$. 

For a positive integer $n$,
the $n$-th power-sum symmetric function will be denoted by $p_n$.
For a partition $\lambda=(\lambda_1,\ldots,\lambda_\ell)$, 
$p_\lambda := p_{\lambda_1} \cdots p_{\lambda_\ell}$
is the the product of power-sum symmetric functions.
Recall that $\{p_\lambda \mid \lambda \vdash d\}$ 
is a basis of $\Lambda_{\bb{Q},d}$.

For two indeterminates $q$ and $t$,
the Macdonald symmetric function will be denoted by $P_{\lambda}(q,t)$,
and the integral form will be denoted by $J_{\lambda}(q,t)$.
These are the elements of $\Lambda_{\bb{F}}$ with 
$\bb{F} := \bb{Q}(q,t)$.
More precisely,
$\{P_{\lambda}(q,t) \mid \lambda \vdash d \}$ 
is a basis of $\Lambda_{\bb{F},d}$.


Finally we introduce the $q$-Pochhammer symbols
$$
 (a;q)_\infty := \prod_{i=0}^{\infty}(1-a q^i), \qquad
 (a;q)_n := \dfrac{(a;q)_\infty}{(a q^n;q)_\infty}
$$

\subsection*{Acknowledgements}

The author is supported by the Grant-in-aid for 
Scientific Research (No.\ 25800014), JSPS.

This work is also partially supported by the 
JSPS Strategic Young Researcher 
Overseas Visits Program for Accelerating Brain Circulation
``Deepening and Evolution of Mathematics and Physics,
Building of International Network Hub based on OCAMI''.

Main part of this work was done while the author 
stayed at University of Toronto (Autumn 2013)
and National Research University  Higher School of Economics (Winter 2013-14).
The author would like to thank both institutes for support and hospitality.

Finally, part of the results was presented in a talk in the workshop ``Quiver varieties" at 
Simons Center for Geometry and Physics (mid-October 2013).
The author would like to thank the organizers of the workshop for the invitation.

\section{Awata-Yamada conjecture for Whittaker vector}
\label{sec:statement}

\subsection{Deformed Virasoro algebra}

Let us begin with the introduction of the deformed Virasoro algebra \cite{SKAO}.
Let $q$ and $t$ be indeterminates and set $\bb{F} := \bb{Q}(q,t)$.
Consider the associative algebra over $\bb{F}$
generated by $\{T_n \mid n \in \bb{Z}\}$ and $1$ 
with the relations
\begin{align}
\label{eq:qvir:rel}
 \sum_{l \ge 0} f_l (T_{m-l} T_{n+l} - T_{n-l} T_{m+l})
=-\dfrac{(1-q)(1-t^{-1})}{1-q/t}
  \bigl((q/t)^m - (q/t)^{-m}\bigr) \delta_{m+n,0}
\end{align}
for any $m,n \in \bb{Z}$.
Here the coefficients $f_l$'s are given by the generating function
\begin{align*}
  \sum_{l \ge 0} f_l z^l
= \exp\Bigl(
   \sum_{n\ge1} 
    \dfrac{1}{n} \dfrac{(1-q^n)(1-t^{-n})}{1+(q/t)^{n}} z^n 
  \Bigr).
\end{align*}
This associative algebra will be denoted by $\Vir_{q,t}$.

The name ``deformed Virasoro'' comes from a degenerate limit of this algebra.
Let us specialize $q,t$ to complex numbers 
(so that the algebra is defined over $\bb{C}$)
and set $q = e^\hbar$, $t=q^\beta$. 
Take the limit $\hbar \to 0$ with fixed $\beta$.
Then the current $T(z) := \sum_{n \in \bb{Z}}  T_n z^{-n}$ 
can be expanded as 
$$
 T(z) = 2 + 
  \beta \left(z^2 L(z) +\dfrac{(1-\beta)^2}{4\beta}\right) \hbar^2 +
  O(\hbar^4),
$$
where $L(z) = \sum_{n \in \bb{Z}} L_n z^{-n-2}$
satisfies the defining relation of the Virasoro Lie algebra
$$
 [L_m,L_n] = (m-n)L_{m+n} + \dfrac{c (m^3-m)}{12}\delta_{m+n,0}
$$
with 
$c=1-6(1-\beta)^2/ \beta$.

Let $h$ be an element of $\bb{F}$.
The Verma module $M_h$ of $\Vir_{q,t}$ 
is generated by the highest weight vector 
$\vac{h}$ with the properties 
$$
 T_0\vac{h}=h\vac{h}, \qquad 
 T_n\vac{h}=0 \ (n \ge 1).
$$

By the defining relation \eqref{eq:qvir:rel},
$\Vir_{q,t}$ has a natural decomposition.
Introduce the outer grading operator $d$ 
which satisfies $[d,T_n]=n T_n$ and $d \vac{h}= 0$.
Let us call an element $v \in M_h$ of grade $n$ if $d v = -n v$.
Then one has the decomposition 
$M_h = \oplus_{n \ge 0} M_{h,n}$ 
by the grade $n$ subspace $M_{h,n}$. 
Let us set $\widehat{M}_h := \prod_{n\ge0}M_{h,n}$, 
the completion of $M_h$ with respect to this grading decomposition.
So an element of $\widehat{M}_h$ can be written 
as an infinite sum $\sum_{n \ge 0} v_n$ with $v_n \in M_{h,n}$.

\subsection{Whittaker vector}

Let $L$ be an element of $\bb{F}$.
An element $v_G$ of the completed Verma module $\widehat{M}_h$ 
with the properties
\begin{align}
\label{eq:Whittaker}
 T_1 v_G = L v_G, \qquad 
 T_n v_G = 0 \quad (n \ge 2)
\end{align}
will be called a Whittaker vector of $\Vir_{q,t}$.
This notion was introduced in \cite{AY},
where $v_G$ was called ``the (deformed) Gaiotto state''.
$v_G$ is an analogue of the (degenerate) Whittaker vector $w$ 
for Virasoro Lie algebra, 
and $w$ was studied in \cite{G} from the viewpoint of AGT conjecture.

\subsection{Bosonization}

Let us recall the bosonization of $\Vir_{q,t}$ following \cite{SKAO}.
Consider the Heisenberg Lie algebra $\cal{H}$ over $\bb{F}$
generated by $\{a_n \mid n \in \bb{Z}\}$ and $1$ 
with the commutation relations
$$
 [a_m,a_n] = m \dfrac{1-q^{|m|}}{1-t^{|m|}}\delta_{m+n,0}.
$$
Then the $\Vir_{q,t}$ current 
$T(z) = \sum_{n \in \bb{Z}} T_n z^{-n}$ 
is bosonized as 
\begin{align}
\label{eq:qvir:bosonization}
& T(z) = \Lambda^{+}\bigl( (q/t)^{-1/2}z \bigr) \cdot (q/t)^{1/2}t^{a_0}
       + \Lambda^{-}\bigl( (q/t)^{ 1/2}z \bigr) \cdot (q/t)^{-1/2}t^{-a_0},
\\
\nonumber
& \Lambda^{\pm}(z) 
  := \exp\Bigl(
         \pm \sum_{n=1}^{\infty}
            \dfrac{1-t^{-n}}{1+(q/t)^n} \dfrac{a_{-n}}{n} z^{n}
         \Bigr) 
     \exp\Bigl(
         \mp \sum_{n=1}^{\infty}
            (1-t^n) \dfrac{a_{n}}{n} z^{-n}
         \Bigr).
\end{align}

The Verma module $M_h$ is also realized 
in the Heisenberg Fock space $\cal{F}$.
The highest weight vector of $\cal{F}$ is denoted by $1_{\cal{F}}$, 
which satisfies the properties
$$
 a_0 1_\cal{F} = \alpha 1_\cal{F}, \qquad
 a_n 1_\cal{F} = 0 \ (n \ge 1). 
$$
Here we set the highest weight of $1_{\cal{F}}$ to be $\alpha$.

Since we take $q$ and $t$ to be indeterminates, 
the map 
\begin{align}
\label{eq:Fock}
 M_h \longrightarrow \cal{F}
\end{align}
induced by the bosonization \eqref{eq:qvir:bosonization} 
and the correspondence 
$\vac{h} \mapsto 1_\cal{F}$
with 
$$
 h= (q/t)^{1/2}t^\alpha +  (q/t)^{-1/2}t^{-\alpha}
$$
gives an isomorphism of $\Vir_{q,t}$-modules.
If we specialize $q$ and $t$ to generic complex numbers,
the above map is still an isomorphism.
The non-generic values are 
given by the zeros of the ($\Vir_{q,t}$ version of) Kac determinant,
which was conjectures in \cite{SKAO} and proved in \cite{BP}.

When expressing elements of $M_h$ by those of $\cal{F}$,
it is convenient and powerful to use the symmetric function.
An element of $\cal{F}$ can be presented as a symmetric function 
via the correspondences
$$
 a_{-n} \longmapsto p_n,\qquad 
 a_n \longmapsto n\dfrac{1-q^n}{1-t^n}\partial_{p_n}
$$
for $n>0$.
Thus an element 
$a_{-\lambda_1} \cdots a_{-\lambda_\ell} \cdot 1_{\cal{F}}$
of the PBW-basis of $\cal{F}$
can be identified with the  power-sum symmetric function
$p_\lambda = p_{\lambda_1} \cdots p_{\lambda_\ell}$.
By this correspondence, we have an isomorphism 
\begin{align*}
 \cal{F} \xrightarrow{\ \sim \ } \Lambda_{\bb{F}}
\end{align*}
of $\bb{F}$-vector spaces.
Composed with \eqref{eq:Fock}, it gives an isomorphism
$$
 M_h \xrightarrow{\ \sim \ } \Lambda_{\bb{F}}
$$
of $\Vir_{q,t}$-modules.
The action of $T(z)$ on $\Lambda_{\bb{F}}$
can be written as 
\begin{equation}
\label{eq:Tz:F}
\begin{split}
T(z) = 
&\exp\Bigl(
       \sum_{n=1}^{\infty}
        \dfrac{1-t^{-n}}{1+(q/t)^n} \dfrac{p_n}{n} 
        (q/t)^{-n/2} z^{n}
      \Bigr)
  \times
  \exp\Bigl(
      - \sum_{n=1}^{\infty}
         (1-q^n) \partial_{p_n} (q/t)^{n/2} z^{-n}
       \Bigr)
  \times k
\\
+&
  \exp\Bigl(
       -\sum_{n=1}^{\infty}
        \dfrac{1-t^{-n}}{1+(q/t)^n} \dfrac{p_n}{n} 
        (q/t)^{n/2} z^{n}
      \Bigr)
  \times
  \exp\Bigl(
      \sum_{n=1}^{\infty}
       (1-q^n) \partial_{p_n} (q/t)^{-n/2} z^{-n}
      \Bigr)
     \times k^{-1}
\end{split}
\end{equation}
with 
$$k:= (q/t)^{1/2}t^{\alpha}.$$
Thus we have $H = k + k^{-1}$.
We shall use this symbol $k$ for indicating the highest weight, instead of using $\alpha$ or $h$.

\subsection{Awata-Yamada conjecture}
Let us recall the conjectural formula \cite[(3.18)]{AY} 
of Whittaker vector of the deformed Virasoro algebra. 

\begin{thm}\label{thm:main}
The Whittaker vector $v_G$ of the deformed Virasoro algebra exists uniquely,
and in the bosonized form it is given by 
$v_G = \sum_{n\ge0} L^n v_{G,n}$  with 
\begin{align}
\label{eq:AY}
v_{G,n} =
  \sum_{\lambda \, \vdash n} P_{\lambda}(q,t) \gamma_\lambda,
\quad
  \gamma_\lambda =  
  \prod_{s \in \lambda}
  \dfrac{k}{1 - k^2 q^{j(s)} t^{-i(s)}}
  \dfrac{q^{a_\lambda(s)}}{1 - q^{1 + a_\lambda(s)} t^{l_\lambda(s)}}
\end{align}
\end{thm}


\section{Preliminary}
\label{sec:prelim}

The purpose of this subsection is a rewriting of the Whittaker condition
in terms of the key objects $\eta(z)$ and $D(z)$.
These are Laurent formal series of operators acting on the Fock space,
and part of the generating currents of the Ding-Iohara-Miki algebra.
$\eta(z)$ has another feature related to Macdonald symmetric functions:
The zero-th mode $\eta_0$ realizes the first-order 
Macdonald difference operator.
The importance of this current have been observed 
in the past studies \cite{SKAO,S,FHHSY,FFJMM}.

\subsection{Rewriting in the integral form}

First we rewrite the Whittaker vector \eqref{eq:AY}
in terms of the integral form of Macdonald symmetric functions.
For that purpose, we shall prepare several combinatorial notations.

For a partition $\lambda=(\lambda_1,\lambda_2,\ldots)$ 
we use the common symbol
$$ 
 n(\lambda) := \sum_i (i-1)\lambda_i.
$$
One can immediately see that 
\begin{align}
 \label{eq:par:n}
 n(\lambda) = \sum_{s \in \lambda} (i(s)-1),\quad 
 n(\lambda') = \sum_{s \in \lambda} (j(s)-1).
\end{align}
Here the running index $s \in \lambda$ means that 
$s$ runs over the set of boxes in the Young diagram (associated to) $\lambda$.
The symbols $i(s)$ and $j(s)$ are the coordinates of the box $s$.
As explained in NOTATION, we have $1 \le i(s) \le \ell(\lambda)$ and 
$1 \le j(s) \le \lambda_{i(s)}$ for any $s \in \lambda$.

Then the relationship  between 
$P_\lambda(q,t)$ and $J_\lambda(q,t)$ is 
\begin{align}
\label{eq:JPQ}
 J_\lambda(q,t) = c_\lambda(q,t) P_\lambda(q,t) 
 = c'_\lambda(q,t) Q_\lambda(q,t). 
\end{align}
(See \cite[Chap VI., \S 8]{M} for the detail.)
Here the symmetric function $Q_\lambda(q,t)$ is the dual of 
$P_\lambda(q,t)$ with respect to the Macdonald inner product
$\langle \cdot,\cdot\rangle_{q,t}$ on $\Lambda_{\bb{F}}$ defined by 
\begin{align}\label{eq:ipqt}
 \langle p_\lambda,p_\mu \rangle_{q,t} := 
  \delta_{\lambda,\mu} z_\lambda 
  \prod_{i=1}^{\ell{\lambda}}\dfrac{1-q^{\lambda_i}}{1-t^{\lambda_i}}.
\end{align}
The coefficients $c_{\lambda}(q,t)$ and $c'_{\lambda}(q,t)$ are 
given by the product formula
\begin{align}
\label{eq:cc's}
&c_{\lambda}(s;q,t) := 1-q^{a_{\lambda}(s)  } t^{l_{\lambda}(s)+1},
 \qquad
 c'_{\lambda}(s;q,t) := 1-q^{a_{\lambda}(s)+1 } t^{l_{\lambda}(s)},
 \qquad
\\
\label{eq:cc'}
&c_{\lambda}(q,t) := \prod_{s \in \lambda} c_{\lambda}(s;q,t), 
 \qquad
 c'_{\lambda}(q,t) := \prod_{s \in \lambda} c'_{\lambda}(s;q,t).
\end{align}
We also have
\begin{align*}
 \langle J_\lambda(q,t), J_\mu(q,t) \rangle_{q,t} 
 = c_{\lambda}(q,t) c'_{\lambda}(q,t) \delta_{\lambda,\mu}
\end{align*}

We also use the plethystic notation for symmetric functions.
For an alphabet $X=(x_1,x_2,\ldots)$, 
a symmetric function with variables $X$ 
is denoted by $f[X]$.
So, for example, the power symmetric function in the variables $X$ is denoted by 
$p_n[X] := x_1^n + x_2^n + \cdots$.

We shall denote by 
\begin{align*}
H[X;u] := \prod_{i}(1-u X_i)^{-1} = \sum_{r\ge0} u^r h_r[X] 
\end{align*}
the generating function of 
the complete symmetric polynomials $h_r$ in the alphabet $X$.

Following the ideas in \cite{BGHT}, 
for each partition $\lambda$ we introduce the set 
\begin{align}
\label{eq:B}
 B_\lambda 
 := \bigl\{q^{j-1}t^{1-i} \mid 
           1\le i \le \ell(\lambda), \, 1\le j \le \lambda_i \bigr\}
  = \bigl\{q^{j(s)-1}t^{1-i(s)} \mid s\in \lambda \bigr\}.
\end{align}
We shall use $B_\lambda$ as an alphabet of symmetric polynomials.
For example, we have
\begin{align*}
 e_1[B_\lambda] = \sum_{s \in \lambda} q^{j(s)-1}t^{1-i(s)}.
\end{align*}

Now using \eqref{eq:par:n}, \eqref{eq:JPQ} and \eqref{eq:B}, 
we can rewrite the conjectural Whittaker vector formula \eqref{eq:AY} as follows:

\begin{lem}\label{lem:u}
\eqref{eq:AY} is equivalent to 
$v_{G,n} = k^n  \sum_{r=0}^\infty (k^2 q/t)^r v_n^{(r)}$ with 
\begin{align}\label{eq:AY:J}
& v_n^{(r)} := \sum_{\lambda \,\vdash n} 
 J_{\lambda}(q,t) 
  \dfrac{q^{n(\lambda')}}{\langle J_\lambda(q,t),J_\lambda(q,t) \rangle_{q,t}}
 h_r[B_\lambda].
\end{align}
\end{lem}

\subsection{Rewriting via the currents $\eta(Z)$ and $D(z)$}

In this subsection we rewrite the statement in terms of the currents 
$\eta(Z)$ and $D(z)$, which will play important roles in our computation.
Let us introduce the following Laurent series valued in operators acting 
on the Heisenberg Fock space $\cal{F}$.
\begin{align*}
\eta(z) := 
&\exp\Bigl( \sum_{n>0} \dfrac{1-t^{-n}}{n}a_{-n} z^{n} \Bigr)
 \exp\Bigl(-\sum_{n>0} \dfrac{1-t^{n} }{n}a_n    z^{-n}\Bigr)
\\
=
&\exp\Bigl( \sum_{n>0} \dfrac{1-t^{-n}}{n}p_n  z^{n} \Bigr)
 \exp\Bigl(-\sum_{n>0} (1-q^{n})\partial_{p_n} z^{-n}\Bigr),
\\
D(z) :=
&\exp\Bigl(\sum_{n>0} \dfrac{1-t^{n} }{n}a_n  z^{-n}\Bigr)
=\exp\Bigl(\sum_{n>0} (1-q^{n})\partial_{p_n} z^{-n}\Bigr).
\end{align*}
We denote their Fourier expansions as 
$\eta(z) = \sum_{d\in \bb{Z}}\eta_d z^{-d}$ and
$D(z)    = \sum_{d \ge 0}D_d z^{-d}$.
Finally, set 
\begin{align}\label{eq:vnu}
 v_n(u) := \sum_{r=0}^\infty u^r v_n^{(r)},
\end{align}
so that we have 
$v_{G,n}= k^n v_n(k^2 q/t)$.

\begin{prop}\label{prop:whit}
The Whittaker condition \eqref{eq:Whittaker} is equivalent to the condition
\begin{align}
\label{eq:Whittaker:2}
\bigl((q/t)^{d-1} u \eta_d + D_d\bigr)v_n(u) =
\begin{cases} v_{n}(u) & (d=1) \\ 0 & (d \ge2)\end{cases}
\end{align}
for each $n \in \bb{Z}_{\ge0}$ and for any indeterminate $u$.
\end{prop}

\begin{proof}
First we may eliminate the variable $L$ of the vector $v_{G}$ since
\begin{align*}
 T_d v_{G}=0 \Longleftrightarrow  T_d v_{G,n}=0 \ (\forall n \in \bb{Z}_{\ge0})
\end{align*}
for each  $d\in \bb{Z}_{\ge2}$ and
\begin{align*}
 T_1 v_{G}=L v_{G} \Longleftrightarrow  T_1 v_{G,n+1}=v_{G,n} \ (\forall n \in \bb{Z}_{\ge0}).
\end{align*}

An easy calculation says the bosonized current $T(z)$ \eqref{eq:Tz:F} satisfies
\begin{align}\label{eq:pT=ed}
\psi(z) T(z) = \eta\bigl((q/t)^{-1/2} z\bigr) k + D\bigl((q/t)^{1/2} z\bigr) k^{-1}
\end{align}
with
\begin{align*}
\psi(z) :=
  \exp\Bigl(
       \sum_{n=1}^{\infty}
        \dfrac{1-t^{-n}}{1+(q/t)^n} \dfrac{p_n}{n} 
        (q/t)^{n/2} z^{n}
      \Bigr).
\end{align*}
Notice that $\psi(z)$ has only the negative parts in its Fourier expansion: 
$\psi(z) = \sum_{d\ge0}\psi_{-d} z^{d}$.
Then taking the coefficients of $z^{-d}$ ($d \in \bb{Z}_{\ge 0}$)
in the both sides of \eqref{eq:pT=ed}, 
we have 
\begin{align}\label{eq:pTd}
\psi_0 T_d + \psi_{-1} T_{d+1}+ \psi_{-2} T_{d+2}+\cdots
=k (q/t)^{d/2} \eta_d + k^{-1} (q/t)^{-d/2} D_d.
\end{align}

Now we may prove the statement as follows.
Let us fix an $n \in \bb{Z}_{\ge 0}$.
Since $T_d$, $\eta_d$ and $D_d$ act by 0 
on the component $v_{G,n}$ with $n<d$,
it is enough calculate $T_d v_{G,n}$ with$1\le d \le n$.

Using \eqref{eq:pTd} with $d=n$ and noticing $\psi_0=1$, we have 
\begin{align*}
 T_n v_{G,n}=0  
&\Longleftrightarrow 
 \bigl((q/t)^{n/2}k \eta_n +  (q/t)^{-n/2}k^{-1} D_n\bigr)v_{G,n}=0
\\
&\Longleftrightarrow 
 \bigl( (q/t)^{n} k^2 \eta_n + D_n\bigr)v_{G,n}=0
\\
&\Longleftrightarrow 
 \bigl((q/t)^{n-1} u \eta_n + D_n\bigr)v_{n}=0.
\end{align*}
In the last line we changed the notation $k^2 q/t = u$ as in Lemma \ref{lem:u}.
Thus $T_n v_{G,n}=0$ follows from \eqref{eq:Whittaker:2}.

Next by \eqref{eq:pTd} with $d=n-1$  we have 
\begin{align*}
 \bigl((q/t)^{d-1} u \eta_d + D_d\bigr)v_{n}=0\ (d=n,n-1)
&\Longleftrightarrow
 (T_{n-1} + \psi_{-1} T_{n})v_{G,n}=0,\quad T_n v_{G,n}=0
\\
&\Longleftrightarrow
 T_{n-1} v_{G,n}=0,\quad T_n v_{G,n}=0.
\end{align*}
The last line follows from the fact that the operator $\psi_{-1}$ is injective.
A similar argument gives 
\begin{align*}
 \bigl((q/t)^{d-1} u \eta_d + D_d\bigr)v_{n}=0\ (d\ge 2)
\Longleftrightarrow
 T_{d} v_{G,n}=0\ (d\ge 2).
\end{align*}

At last, from \eqref{eq:pTd} with $d=1$ we have 
\begin{align*}
\eqref{eq:Whittaker:2}
&\Longleftrightarrow
(T_1 + \psi_{-1} T_{2}+\cdots+ \psi_{-n+1} T_{n})v_{G,n}=0,\quad 
T_{d}v_{G,n}=0\ (d\ge2)
\\
&\Longleftrightarrow
\eqref{eq:Whittaker}.
\end{align*}
\end{proof}

This Proposition and the definition \eqref{eq:vnu} gives

\begin{cor}\label{cor:whit}
The Whittaker condition \eqref{eq:Whittaker} is equivalent to the condition
\begin{align}
\label{eq:Whittaker:3:1}
&\bigl( u \eta_1 + D_1\bigr)v_{n+1}(u) = v_{n}(u),\\
\label{eq:Whittaker:3:2}
&(q/t)^{d-1} \eta_d v_{n}^{(r-1)} + D_d v_{n}^{(r)}= 0 \quad (d \ge2,\ r \ge 0)
\end{align}
for each $n \ge0$. 
Here we set $v_{n}^{(-1)} := 0$ for any $n$.
\end{cor}

\subsection{Relation with Ding-Iohara-Miki algebra}
\label{ssec:DIM}

Let us make a brief comment about the current $\eta(z)$ and 
the Ding-Iohara-Miki algebra $\cal{U}(q,t)$ following \cite{FHHSY}.
This subsection is not necessary for the logical step of our proof,
but we think it will makes our picture more transparent.

Ding-Iohara algebra \cite{DI:1997} was introduced 
as a generalization of Drinfeld realization of the quantum affine algebra. 
In \cite[Appendix A]{FHHSY}, 
the authors introduced a version $\cal{U}(q,t)$ of the Ding-Iohara algebra 
having two parameters $q$ and $t$. 
The same algebra appeared in the work of Miki \cite{Mi},
so let us call $\cal{U}(q,t)$ the Ding-Iohara-Miki algebra.

Set $\widetilde{\bb{F}}:=\bb{Q}(q^{\pm1},t^{\pm1})$.
$\cal{U}(q,t)$ is a unital associative algebra over $\widetilde{\bb{F}}$
generated  by the Drinfeld currents 
\begin{align*}
x^\pm(z)=\sum_{n\in \bb{Z}}x^\pm_n z^{-n},\qquad 
\psi^\pm(z)=\sum_{\pm m \ge 0}\psi^\pm_m z^{-m},
\end{align*}
and the central element $\gamma^{\pm 1/2}$, satisfying the defining relations
\begin{align*}
&\psi^\pm(z) \psi^\pm(w)= \psi^\pm(w) \psi^\pm(z),
\\
&\psi^+(z)\psi^-(w)=
 \dfrac{g(\gamma^{+1} w/z)}{g(\gamma^{-1}w/z)}\psi^-(w)\psi^+(z),
\\
&\psi^+(z)x^\pm(w)=g(\gamma^{\mp 1/2}w/z)^{\mp1} x^\pm(w)\psi^+(z),
\\
&\psi^-(z)x^\pm(w)=g(\gamma^{\mp 1/2}z/w)^{\pm1} x^\pm(w)\psi^-(z),
\\
&[x^+(z),x^-(w)]=\dfrac{(1-q)(1-1/t)}{1-q/t}
\big( \delta(\gamma^{-1}z/w)\psi^+(\gamma^{1/2}w)-
\delta(\gamma z/w)\psi^-(\gamma^{-1/2}w) \big)
\\
&G^{\mp}(z/w)x^\pm(z)x^\pm(w)=G^{\pm}(z/w)x^\pm(w)x^\pm(z),
\end{align*}
with  
\begin{align*}
g(z) := \dfrac{G^+(z)}{G^-(z)},\qquad
G^\pm(z) := (1-q^{\pm1}z)(1-t^{\mp 1}z)(1-q^{\mp1}t^{\pm 1}z).
\end{align*}
In \cite{FFJMM} the authors noticed that 
one may add one more relation, which is an analogue of the Serre relation.
They also indicated that the Fock representation given below automatically 
satisfies the Serre-like relation.
Thus in the following discussion we can safely ignore it.

The algebra $\cal{U}(q,t)$ has a Hopf algebra structure whose 
coproduct and antipode are written explicitly in terms of 
the above generating currents.
Since those formulae are not necessary in this paper, 
we omit them 
(see \cite[Proposition A.2]{FHHSY} for the detail).

We say a representation of $\cal{U}(q,t)$ is of level $k$, 
if the central element $\gamma$ is realized by the constant 
$(t/q)^{k/2}=p^{-k/2}$.
We can construct a Fock representation $\rho_c(\cdot)$ of level one as follows.

Let us also introduce the following four vertex operators 
\cite[(1.7),(3.23),(3.27),(3.28)]{FHHSY}.
\begin{align*}
&\eta(z) := 
\exp\Big( \sum_{n>0} \dfrac{1-t^{-n}}{n}a_{-n} z^{n} \Big)
\exp\Big(-\sum_{n>0} \dfrac{1-t^{n} }{n}a_n    z^{-n}\Big),
\\
&\xi(z) :=
\exp\Big(-\sum_{n>0} \dfrac{1-t^{-n}}{n}p^{-n/2}a_{-n} z^{n}\Big)
\exp\Big(\sum_{n>0}  \dfrac{1-t^{n}}{n} p^{-n/2} a_n z^{-n}\Big),
\\
&\varphi^+(z) :=
\exp\Big(-\sum_{n>0} \dfrac{1-t^{n}}{n} (1-p^{-n})p^{n/4} a_n z^{-n}
    \Big),
\\
&\varphi^-(z) :=
\exp\Big(\sum_{n>0} \dfrac{1-t^{-n}}{n}(1-p^{-n})p^{n/4} a_{-n} z^{n}
     \Big).
\end{align*}

Then for a fixed $c\in \widetilde{\bb{F}}^{\times}$,
we have a level one representation $\rho_c(\cdot)$ of $\cal{U}(q,t)$ 
on our Fock space $\cal{F}$  (see the paragraph above \eqref{eq:Fock})
by setting
\begin{align*}
\rho_c(\gamma^{\pm 1/2})=(q/t)^{\mp 1/4},\quad
\rho_c(\psi^\pm(z))=\varphi^\pm(z),\quad 
\rho_c(x^+(z))=c\, \eta(z),\quad
\rho_c(x^-(z))=c^{-1} \xi(z).
\end{align*}
Here the current $\eta(z)$ appears.

By the work of Schiffmann and Vasserot \cite{SV}, 
the algebra $\cal{U}(q,t)$ has another description.
It is the (central extension of) Drinfeld double of 
the Hall algebra of elliptic curve \cite{BS}.
Due to this Hall algebra description, 
$\cal{U}(q,t)$  has a double-grading structure.
One may take another set of generators 
\begin{align}\label{eq:uqt:SV}
 \{u_{r,d} \mid (r,d)\in\bb{Z}^2\setminus\{(0,0)\}\}
\end{align}
with some explicit relation.
The relationship of the Drinfeld current description and the Hall algebra description 
is as follows:
$x^{\pm}_{n}=u_{-n,\pm1}$, $\psi^{\pm}_{m}=u_{\mp m,0}$.

If one ignores the central extension,
the relation enjoys the $\operatorname{SL}(2,\bb{Z})$,
which comes from the autoequivalence of the derived category of coherent sheaves on 
the elliptic curve.
As a consequence, Schiffmann and Vasserot discovered that 
there are infinite families of Heisenberg subalgebras in $\cal{U}(q,t)$.
For each line $\ell$ through the origin in the $\bb{Z}^2$ lattice \eqref{eq:uqt:SV},
the subalgebra $\cal{U}(q,t)_{\ell}$ 
generated by the elements $\{u_{r,d} \mid (r,d) \in \ell \}$ placed on $\ell$ 
satisfies the Heisenberg relation.

In our discussion, the subalgebra $\cal{U}(q,t)_{\ell_\infty}$ 
of $\cal{U}(q,t)$ generated by the elements $\{u_{(0,d)}\mid d \in \bb{Z}\setminus\{0\}\}$ 
will play an important role.
In this subalgebra, the Heisenberg relation becomes trivial,
so that the subalgebra is commutative. 
It corresponding to the algebra of Macdonald difference operators realized in the Fock space.
We will describe this subalgebra in the next subsection.

By B. Feigin and Tsymabliuk's work \cite{FT} and 
Schiffmann and Vasserot's work \cite{SV},
the Fock representation of $\cal{U}(q,t)$ 
is also related to the equivariant $K$-theory 
of the Hilbert schemes of points on the affine plane.
We will briefly mention their results in \S\ref{ssec:goem:uqt}

\subsection{The Whittaker condition in terms of higher-order Macdonald difference operators}

The current $\eta(z)$ is related to the Macdonald difference operator.
As indicated in \cite{SKAO},
in the Fourier expansion $\eta(z) = \sum_{n\in \bb{Z}}\eta_d z^{-d}$,
the zeroth mode $\eta_0$ gives the Heisenberg realization of the (stable limit of the) 
first-order Macdonald difference operator.
\begin{align}\label{eq:eta0:act}
\eta_0 P_\lambda(q,t)=P_\lambda(q,t) \cdot \bigl(1-(1-q)(1-1/t)e_1[B_\lambda]\bigr).
\end{align}

In \cite{FHHSY} the authors give an explicit description of 
the higher-order Macdonald difference operators and 
their algebra structure in terms of the Feigin-Odesskii shuffle product.
In terms of the Hall algebra description, this subalgebra corresponds to the subalgebra 
generated by $\{u_{0,d} \mid d \in \bb{Z}\setminus \{0\}\}$.

We have explicit formulae for a certain family of bosonized Macdonald difference operators
containing the first-order operator \eqref{eq:eta0:act}.
To give them, let us set 
$$
 \ve(x) := \dfrac{1-x}{1-x/t}.
$$
We consider $\ve(x)$ as a formal series in $x$.

\begin{fct}[{\cite[\S5]{BGHT},\cite[\S9]{S}}]\label{fct:E_r}
For $r \in \bb{Z}_{\ge 1}$, 
set the operator  $\widehat{E}_r$ acting on the Fock space $\cal{F}$ by
\begin{align}\label{eq:E_r}
 \widehat{E}_r := \dfrac{t^{-r(r+1)/2}}{(t^{-1};t^{-1})_r}
 \Bigl[\prod_{1\le i<j \le r} \ve(z_j/z_i) \cdot \nord{\eta(z_1)\eta(z_2)\cdots\eta(z_r)}\Bigr]_1.
\end{align}
Here $[ \,\cdots ]_1$ means taking constant part of the Laurent formal series in $z_i$'s,
and $\nord{\ }$ is the (usual) normal ordering of the Heisenberg algebra $\cal{H}$.
Then we have
\begin{align*}
 \widehat{E}_r J_\lambda(q,t) 
 = J_\lambda(q,t) e_r[s^\lambda]
\end{align*}
for any partition $\lambda$, 
where the alphabet $s^\lambda$ is the set of infinite variables given by 
$$
 s^\lambda := 
 (t^{-1}q^{\lambda_1},t^{-2}q^{\lambda_2},\ldots,t^{-\ell}q^{\lambda_\ell},
  t^{-\ell-1},t^{-\ell-2},\ldots).
$$
Here $\ell$ is the length of the partition $\lambda$.
\end{fct}

For later use, we give a formula 
translating the alphabets $B_\lambda$ and $s^\lambda$.

\begin{lem}\label{lem:B-s}
For any $r \in \bb{Z}_{\ge 1}$ we have 
\begin{align*} 
 h_r[B_\lambda]
= \sum_{k=0}^r (-t)^k \theta_{r-k} 
   \sum_{r_i \ge 1,\, \sum_i r_i = k} 
    q^{\sum_{i}(i-1)r_i} e_{r_1,r_2,\ldots}[s^\lambda]
\end{align*}
with $e_{r_1,r_2,\ldots} := e_{r_1} e_{r_2} \cdots$ and
\begin{align*}
\theta_{n} := \sum_{\lambda \, \vdash n}
 \dfrac{1}{c_\lambda(q^{-1},t^{-1}) c_{\lambda'}(q,t)}.
\end{align*}
\end{lem}

\begin{proof}
By the relation
$$
 e_1[B_\lambda] = \dfrac{1}{(1-q)(1-t^{-1})}-\dfrac{t }{1-q}e_1[s^\lambda],
$$
we have
\begin{align*}
\prod_{s\in\lambda}(1-u B_s)^{-1} 
= \prod_{i,j\ge0} (1-u q^i t^{-j})^{-1}
  \prod_{i,j\ge0} (1+u t q^j s^{\lambda}_i).
\end{align*}
Taking the coefficients of $u^r$ in both sides,
we have 
\begin{align*} 
 h_r[B_\lambda]
= \sum_{k=0}^r (-t)^k h_{r-k}[B_\infty]  
   \sum_{r_i \ge 1,\, \sum_i r_i = k} 
    q^{\sum_{i}(i-1)r_i} e_{r_1,r_2,\ldots}[s^\lambda].
\end{align*}
Here we introduced the new alphabet 
\begin{align*}
 B_\infty := \{q^{i} t^{-j} \mid i,j \in \bb{Z}_{\ge 0} \}
\end{align*}
consisting of infinite variables.

Now we compute $h_r[B_\infty]$.
Recall the Cauchy kernel of Macdonald symmetric functions 
(see \cite[Chap. VI, \S2]{M} for the detail):
\begin{align*}
\Pi(X,Y;q,t) := 
 \exp\Bigl(\sum_{n>0} \dfrac{1}{n}\dfrac{1-t^n}{1-q^n}p_n[X]p_n[Y]\Bigr).
\end{align*}
It enjoys the formula
\begin{align}\label{eq:Cauchy}
\Pi(X,Y;q,t)
 = \sum_{\lambda} \dfrac{J_\lambda[X;q,t]J_\lambda[Y;q,t]}{
    \langle J_\lambda(q,t),J_\lambda(q,t) \rangle_{q,t}}.
\end{align}
Here we denote by $P_\lambda[X;q,t]$ the Macdonald symmetric function $P_\lambda(q,t)$ 
with the alphabet $X$ as variables.
$Q_\lambda[X;q,t]$ and  $J_\lambda[X;q,t]$ are similarly defined.

Let us also recall the specialization $\ve_{u,t}$ \cite[Chap VI, \S6]{M} 
of Macdonald symmetric functions.
Let $u$ be an indeterminate.
Under the homomorphism $\ve_{u,t}:\Lambda_{\bb{F}} \to \bb{F}[u]$ defined by
\begin{align*}
\ve_{u,t} p_r  := \dfrac{1-u^r}{1-t^r},
\end{align*}
we have
\begin{align}\label{eq:spec_ve}
\ve_{u,t} J_\lambda(q,t) = 
 \prod_{s \in \lambda} (t^{i(s)-1}-q^{j(s)-1} u).
\end{align}

Then, by applying the homomorphism $\ve_{0,t}:p_r\mapsto 1/(1-t^r)$ 
on the alphabets $X$ and $Y$ in \eqref{eq:Cauchy}, 
we have
$$
  \exp\Bigl(\sum_{n>0} \dfrac{1}{n(1-q^n)(1-t^n)} \Bigr)
=\sum_{\lambda} \dfrac{t^{2n(\lambda)}}{c_\lambda(q,t)c_{\lambda'}(q,t)}
$$
Replacing $t$ with $t^{-1}$ and taking the degree $n$ part,
we have the consequence.
\end{proof}

Following \cite{BGHT}, we define for each symmetric function $f$ 
a version of Macdonald operator
$\Delta_f: \Lambda_{\bb{F}} \to \Lambda_{\bb{F}}$
by 
\begin{align}\label{eq:Delta_f}
\Delta_f J_\lambda(q,t) := J_\lambda(q,t) f[B_\lambda].
\end{align}
Here we used the alphabet $B_\lambda$ (see \eqref{eq:B} for the definition).
By the action \eqref{eq:eta0:act} of $\eta_0$, we have 
\begin{align}\label{eq:eta0Delta}
1- \eta_0 = (1-q)(1-t^{-1}) \Delta_{e_1}.
\end{align}

Recalling the definition \eqref{eq:AY:J} of $v_n^{(r)}$, we have 
$$
 v_n^{(r)} = \Delta_{h_r} v_n^{(0)}.
$$
with $\Delta_{h_{-1}} := 0$.
Then we can rewrite the condition \eqref{eq:Whittaker:3:2}
into the following form:

\begin{cor}\label{cor:whit:2}
The Whittaker condition \eqref{eq:Whittaker} is equivalent to the condition
\begin{align}
\label{eq:main:1}
&\bigl( u \eta_1 + D_1\bigr)v_{n+1}(u) = v_{n}(u),\\
\label{eq:main:2}
&D_d v_n^{(0)} = 0,
\\
\label{eq:main:3}
&\bigl((q/t)^{d-1} \eta_d \Delta_{h_{r-1}}+ D_d \Delta_{h_{r}}\bigr) v_{n}^{(0)} = 
 0 \quad  (d \ge 2,\ r \ge 0)
\end{align}
for each $n \ge 0$. 
\end{cor}

In the next section we give a proof of these relations.
\eqref{eq:main:1} is shown in \S \ref{subsec:part1}.
In \S \ref{subsec:part2} we show \eqref{eq:main:2}.
Finally in \S \ref{subsec:part3} we show \eqref{eq:main:3}.

\section{Proof}
\label{sec:prf}

\subsection{Part 1 of the proof}
\label{subsec:part1}

In this subsection we show the relation
\begin{align}\label{eq:part2}
 \bigl( u \eta_1 + D_1\bigr)v_{n+1}(u) = v_{n}(u)
\end{align}
in Corollary \ref{cor:whit:2}.
Let us introduce a few symbols.
For a box $s$ in a Young diagram, the symbols $B_s$ is defined to be
\begin{align}\label{eq:Bs}
 B_s = B_s(q,t) := q^{j(s)-1}t^{1-i(s)}.
\end{align}
Thus the alphabet $B_\lambda$ defined in \eqref{eq:B} can be rewritten as 
$$
 B_\lambda = \{B_s \mid s \in \lambda\}.
$$
We also have
\begin{align}\label{eq:Deltap1}
 p_1[B_\lambda] = e_1[B_\lambda] = \sum_{s\in\lambda}B_s.
\end{align}
Next we rewrite the Pieri formula \cite[Chap VI, \S7]{M} of 
the Macdonald symmetric function into the following form:
\begin{align}\label{eq:Pieri:d}
 p_1 J_\mu(q,t) 
 = \sum_{\lambda} J_\mu(q,t) d_{\lambda/\mu}(q,t),\qquad
 d_{\lambda/\mu}(q,t) := 
  \psi'_{\lambda/\mu}(q,t) \dfrac{c_{\mu}(q,t)}{c'_{\lambda}(q,t)}.
\end{align}
The running index $\lambda$ runs over the set of partitions such that 
$|\lambda|= |\mu| +1$ and $\lambda \supset \mu$.
See \eqref{eq:cc'} 
for the factor $c'_{\lambda}(q,t)$. 
The factor $\psi'_{\lambda / \mu}(q,t)$ is defined by
\begin{align}\label{eq:psi'}
 \psi'_{\lambda / \mu}(q,t) := 
 \prod_{s \in C_{\lambda / \mu} - R_{\lambda / \mu}}
  \dfrac{b_{\lambda}(s;q,t)}{b_{\mu}(s;q,t)},
\qquad
b_{\lambda}(s;q,t) := 
 \begin{cases}
  \dfrac{c_{\lambda}(s;q,t)}{c'_{\lambda}(s;q,t)} & s \in \lambda
  \\ 
  1 & s \not\in \lambda
 \end{cases}.
\end{align}
Here 
$C_{\lambda / \mu}$ (resp. $R_{\lambda / \mu}$)
is the set of columns (resp. rows) that intersect with 
the boxes in $\lambda / \mu$.

Using the Macdonald inner product \eqref{eq:ipqt}, we also have
\begin{align}\label{eq:Pieri:c}
 \partial_{p_1} J_\lambda(q,t) 
 = \sum_{\mu} J_\mu(q,t) c_{\lambda/\mu}(q,t),\qquad
 c_{\lambda/\mu}(q,t) :=   
   \psi'_{\lambda/\mu}(q,t) \dfrac{c'_{\lambda}(q,t)}{c'_{\mu}(q,t)} \dfrac{1-t}{1-q}.
\end{align}
The running index $\mu$ runs over the set of partitions such that 
$|\mu| = |\lambda|-1$ and $\mu \subset \lambda$.

\begin{lem}
We have
\begin{align}\label{eq:eta1}
\eta_1 J_\lambda(q,t) 
 = -(1-q) \sum_{\mu} J_{\mu}(q,t) 
     c_{\lambda/\mu}(q,t) B_{s_{\lambda/\mu}}.
\end{align}
The range of $\mu$ is the same as in \eqref{eq:Pieri:c}.
The symbol $s_{\lambda/\mu}$ is 
the unique box belonging to the skew diagram $\lambda/\mu$. 
\end{lem}

\begin{proof}
By the definition 
$$
 \eta(z) = 
 \exp\Big( \sum_{n>0} \dfrac{1-t^{-n}}{n}p_n  z^{n} \Big)
 \exp\Big(-\sum_{n>0} (1-q^{n})\partial_{p_n} z^{-n}\Big)
$$
of $\eta(z)$, we have
$$
 \eta_1 = (1-t^{-1})^{-1}[\partial_{p_1},\eta_0].
$$
We also have $\eta_0 = 1-(1-q)(1-t^{-1})\Delta_{p_1}$ \eqref{eq:eta0Delta}.
Then
\begin{align*}
 \eta_1 J_\lambda(q,t) 
&= -(1-q) [\partial_{p_1},\Delta_{p_1}] J_\lambda(q,t) 
\\
&= -(1-q) 
   \sum_{\mu} J_\mu(q,t) c_{\lambda/\mu}(q,t) \bigl(p_1[B_\lambda]-p_1[B_\mu]\bigr)
\\
&= -(1-q) \sum_{\mu} J_\mu(q,t) c_{\lambda/\mu}(q,t) B_{s_{\lambda/\mu}}.
\end{align*}
At the last line we used \eqref{eq:Deltap1}.
This is the desired formula.
\end{proof}

\begin{rmk}
Using the nabla operator introduced in \cite{BGHT},
we can rewrite \eqref{eq:eta1} in a nice form.
Let $\nabla:\Lambda_{\bb{F}} \to \Lambda_{\bb{F}}$ be  the operator 
defined by
\begin{align}\label{eq:nabla}
 \nabla J_\lambda(q,t) 
 := \Delta_{e_{|\lambda|}} J_\lambda(q,t)
  = J_\lambda(q,t) \prod_{s \in \lambda} B_s
  = J_\lambda(q,t) q^{n(\lambda')}t^{-n(\lambda)}.
\end{align}
Then we have
\begin{align}\label{eq:eta-nabla}
\eta_1 = -(1-q)\nabla^{-1} \partial_{p_1} \nabla.
\end{align}
This formula appears in \cite[I.12 (iii)]{BGHT}.
\end{rmk}

Now we start 
\begin{proof}[{Proof of \eqref{eq:part2}}]
Recall
$$
 v_n(u) = \sum_{r=0}^\infty u^r v_n^{(r)}
 =\sum_{\lambda\,\vdash n} J_\lambda(q,t)
  \dfrac{q^{n(\lambda')}}{\langle J_\lambda(q,t),J_\lambda(q,t) \rangle_{q,t}}
  H[B_\lambda;u].
$$
Then by \eqref{eq:eta1} we have
\begin{align*}
\eta_1 v_n(u)
&=-(1-q) \sum_{\lambda \, \vdash n} 
   \dfrac{q^{n(\lambda')}}{\langle J_\lambda(q,t),J_\lambda(q,t) \rangle_{q,t}}
   H[B_\lambda;u]
   \sum_{\substack{\mu\, \vdash n \\ \mu \subset \lambda}}   
    J_\mu(q,t) c_{\lambda/\mu}(q,t) B_{s_{\lambda/\mu}}.
\\
&=-(1-q) \sum_{\mu \, \vdash n-1}  J_\mu(q,t)  
   \dfrac{q^{n(\mu')}}{\langle J_\mu(q,t),J_\mu(q,t) \rangle_{q,t}}
   H[B_\mu;u]
\\
& \qquad \qquad \qquad \qquad \times
   \sum_{\substack{\lambda\, \vdash n \\ \lambda \supset \mu} } 
    q^{j(s_{\lambda/\mu})-1} c_{\lambda/\mu}(q,t) 
    \dfrac{c_\mu(q,t) c'_\mu(q,t)}{c_\lambda(q,t) c'_\lambda(q,t)}
    \dfrac{B_{s_{\lambda/\mu}}}{1-u B_{s_{\lambda/\mu}}}.
\\
&=-(1-t) \sum_{\mu \, \vdash n-1}  J_\mu(q,t)  
   \dfrac{q^{n(\mu')}}{\langle J_\mu(q,t),J_\mu(q,t) \rangle_{q,t}}
   H[B_\mu;u]
 \sum_{\substack{\lambda\, \vdash n \\ \lambda \supset \mu} } 
  q^{j(s_{\lambda/\mu})-1} d_{\lambda/\mu}(q,t) 
    \dfrac{B_{s_{\lambda/\mu}}}{1-u B_{s_{\lambda/\mu}}}.
\end{align*}
At the last line we used the definitions \eqref{eq:Pieri:c}, \eqref{eq:Pieri:d} of 
$c_{\lambda/\mu}(q,t)$, $d_{\lambda/\mu}(q,t)$.

Next we compute $D_1 v_n(u)$.
By $D_1 = (1-q) \partial_{p_1}$ and \eqref{eq:Pieri:d}, we have
\begin{align*}
 D_1 v_n(u) 
&= (1-q)\partial_{p_1} 
  \sum_{\lambda\,\vdash n} J_\lambda(q,t)
   \dfrac{q^{n(\lambda')}}{\langle J_\lambda(q,t),J_\lambda(q,t) \rangle_{q,t}}
   H[B_\lambda;u]
\\
&= (1-q) \sum_{\lambda \, \vdash n} 
   \dfrac{q^{n(\lambda')}}{\langle J_\lambda(q,t),J_\lambda(q,t) \rangle_{q,t}}
   H[B_\lambda;u]
   \sum_{\substack{\mu\, \vdash n \\ \mu \subset \lambda}}   
    J_\mu(q,t) c_{\lambda/\mu}(q,t) 
\\
&= (1-q) \sum_{\mu \, \vdash n-1}  J_\mu(q,t)  
   \dfrac{q^{n(\mu')}}{\langle J_\mu(q,t),J_\mu(q,t) \rangle_{q,t}}
   H[B_\mu;u]
\\
& \qquad \qquad \qquad \qquad \times
   \sum_{\substack{\lambda\, \vdash n \\ \lambda \supset \mu} } 
    q^{j(s_{\lambda/\mu})-1} c_{\lambda/\mu}(q,t) 
    \dfrac{c_\mu(q,t) c'_\mu(q,t)}{c_\lambda(q,t) c'_\lambda(q,t)}
    \dfrac{1}{1-u B_{s_{\lambda/\mu}}}.
\\
&= (1-t) \sum_{\mu \, \vdash n-1}  J_\mu(q,t)  
   \dfrac{q^{n(\mu')}}{\langle J_\mu(q,t),J_\mu(q,t) \rangle_{q,t}}
   H[B_\mu;u]
 \sum_{\substack{\lambda\, \vdash n \\ \lambda \supset \mu} } 
  q^{j(s_{\lambda/\mu})-1} d_{\lambda/\mu}(q,t) 
    \dfrac{1}{1-u B_{s_{\lambda/\mu}}}.
\end{align*}

Combining these calculations, we have
\begin{align*}
(u \eta_1 + D_1) v_n(u) 
&= (1-t) \sum_{\mu \, \vdash n-1}  J_\mu(q,t)  
   \dfrac{q^{n(\mu')}}{\langle J_\mu(q,t),J_\mu(q,t) \rangle_{q,t}}
   H[B_\mu;u]
 \sum_{\substack{\lambda\, \vdash n \\ \lambda \supset \mu} } 
  q^{j(s_{\lambda/\mu})-1} d_{\lambda/\mu}(q,t).
\end{align*}

Now we want to calculate the summation over $\lambda$, fixing a partition $\mu$.
Using Lemma \ref{lem:dq} below, we have 
\begin{align}\label{eq:dq0}
 (1-t) \sum_{\substack{\lambda\, \vdash n \\ \lambda \supset \mu} } 
  q^{j(s_{\lambda/\mu})-1} d_{\lambda/\mu}(q,t) = 1.
\end{align}
Thus we have
\begin{align*}
(u \eta_1 + D_1) v_n(u) 
= \sum_{\mu \, \vdash n-1}  J_\mu(q,t)  
   \dfrac{q^{n(\mu')}}{\langle J_\mu(q,t),J_\mu(q,t) \rangle_{q,t}}
   H[B_\mu;u]
= v_{n-1}(u).
\end{align*}
\end{proof}

Before giving Lemma \ref{lem:dq},
we recall a few notations on the plethystic formula.
(For more explanation, see \cite{BGHT,GT,H:2003}.)
Let $X$ be an alphabet $X$ and $f$ be a symmetric function $f$.
For an indeterminate $z$, the plethystic transformation is defined to be 
\begin{align*}
&f[z X] := f[X]\Bigr|_{p_n[X] \mapsto z^n p_n[X]}.
\end{align*}
We also denote by $\omega$ the classical involution. 
It is defined as 
\begin{align*}
&\omega f[X] := f[X]\Bigr|_{p_n[X] \mapsto (-1)^{n+1} p_n[X]}.
\end{align*}
$\omega$ interchanges the elementary and complete symmetric functions:
$\omega e_r[X] = h_r[X]$.
For Schur functions we have $\omega s_\lambda[X] = s_{\lambda'}[X]$.

Finally for a symmetric function $f$ homogeneous of degree $d$,
we define 
$$
 f[-X] := (-1)^d \omega f[X].
$$

As an application, we can write the generating function of 
the elementary symmetric functions $e_r[Y-X]$ for two alphabets $X=\{X_i\}$ and
$Y=\{Y_j\}$.
\begin{align}\label{eq:eXY}
 \sum_{r \ge 0} u^r e_r[Y-X] = \dfrac{\prod_j(1+ u Y_j)}{\prod_i(1+u X_i)}.
\end{align}

\begin{lem}\label{lem:dq}
For a partition $\mu$ and $r \in \bb{Z}_{\ge 0}$, we have 
\begin{align*}
(1-t)\sum_{\substack{\lambda\, \vdash |\mu|+1 \\ \lambda \supset \mu} } 
  q^{j(s_{\lambda/\mu})-1} d_{\lambda/\mu}(q,t) B_{s_{\lambda/\mu}}^r = 
 e_r[(1-q)(1-t^{-1}) B_\mu -1].
\end{align*}
In particular, setting $r=0$, we have \eqref{eq:dq0}.
\end{lem}

\begin{proof}
Let $m$ be the number of corners of the Young diagram $\mu$.
Mimicking the calculation in \cite[Theorem 2.2]{GT}, 
we will introduce news variables $X_k$ and $Y_k$ 
(analogue of ``Garsia and Tesler's change of variables" in \cite[Appendix A]{SV}). 

We label the attachable boxes of $\mu$ by 
$s_1,s_2,\ldots,s_{m+1}$ from right to left,
and set $\alpha_k := i(s_k)-1$ and $\beta_k := j(s_k)-1$.
Introduce new variables $X_k$  and $Y_k$ by 
\begin{align}\label{eq:XY}
 X_k := q^{\beta_k} t^{-\alpha_k} = B_{s_k} \ (1\le k \le m+1),\qquad
 Y_k := q^{\beta_k} t^{-\alpha_{k+1}}       \ (1\le k \le m).
\end{align}
We show an example in Figure \ref{fig:2}.
In this case $\mu=(8,6,4,4,1)$, $m=4$ and
\begin{align*}
\begin{array}{lllll}
X_1=q^8,        &X_2=q^6 t^{-1}, &X_3=q^4 t^{-2}, &X_4=q^1t^{-4}, &X_5=t^{-5},\\
Y_1=q^8 t^{-1}, &Y_2=q^6 t^{-2}, &Y_3=q^4 t^{-4}, &Y_4=q^1t^{-5}.
\end{array}
\end{align*}

\begin{figure}[htbp]
\centering
{\unitlength 0.1in%
\begin{picture}( 18.0000, 12.0000)(  4.0000,-16.0000)%
%
\special{pn 8}%
\special{pa 400 600}%
\special{pa 400 1600}%
\special{fp}%
\special{pa 600 1600}%
\special{pa 600 600}%
\special{fp}%
\special{pa 800 800}%
\special{pa 800 1600}%
\special{fp}%
\special{pa 1000 1600}%
\special{pa 1000 800}%
\special{fp}%
\special{pa 1200 800}%
\special{pa 1200 1600}%
\special{fp}%
\special{pa 1400 1200}%
\special{pa 1400 1600}%
\special{fp}%
\special{pa 1600 1200}%
\special{pa 1600 1600}%
\special{fp}%
\special{pa 1800 1400}%
\special{pa 1800 1600}%
\special{fp}%
\special{pa 2000 1400}%
\special{pa 2000 1600}%
\special{fp}%
\special{pa 400 1600}%
\special{pa 2000 1600}%
\special{fp}%
\special{pa 2000 1400}%
\special{pa 400 1400}%
\special{fp}%
\special{pa 400 1200}%
\special{pa 1600 1200}%
\special{fp}%
\special{pa 1200 1000}%
\special{pa 400 1000}%
\special{fp}%
\special{pa 400 800}%
\special{pa 1200 800}%
\special{fp}%
\special{pa 400 600}%
\special{pa 600 600}%
\special{fp}%
%
\special{pn 8}%
\special{pa 400 400}%
\special{pa 600 400}%
\special{pa 600 600}%
\special{pa 400 600}%
\special{pa 400 400}%
\special{dt 0.045}%
%
\special{pn 8}%
\special{pa 600 600}%
\special{pa 800 600}%
\special{pa 800 800}%
\special{pa 600 800}%
\special{pa 600 600}%
\special{dt 0.045}%
%
\special{pn 8}%
\special{pa 1200 1000}%
\special{pa 1400 1000}%
\special{pa 1400 1200}%
\special{pa 1200 1200}%
\special{pa 1200 1000}%
\special{dt 0.045}%
%
\special{pn 8}%
\special{pa 1600 1200}%
\special{pa 1800 1200}%
\special{pa 1800 1400}%
\special{pa 1600 1400}%
\special{pa 1600 1200}%
\special{dt 0.045}%
%
\special{pn 8}%
\special{pa 2000 1400}%
\special{pa 2200 1400}%
\special{pa 2200 1600}%
\special{pa 2000 1600}%
\special{pa 2000 1400}%
\special{dt 0.045}%
\put(20.5000,-15.5000){\makebox(0,0)[lb]{$s_1$}}%
\put(16.5000,-13.5000){\makebox(0,0)[lb]{$s_2$}}%
\put(12.5000,-11.5000){\makebox(0,0)[lb]{$s_3$}}%
\put(6.5000,-7.5000){\makebox(0,0)[lb]{$s_4$}}%
\put(4.5000,-5.5000){\makebox(0,0)[lb]{$s_5$}}%
\end{picture}}%
\caption{Boxes $s_k$ for $\mu=(8,6,4,4,1)$}
\label{fig:2}
\end{figure}
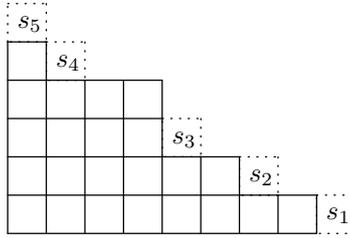

Let $\lambda$ be a partition attaching $s_k$ to $\mu$ 
(so that $s_k = s_{\lambda/\mu}$).
By the definition \eqref{eq:Pieri:d} of $d_{\lambda/\mu}(q,t)$, 
we have
\begin{align*}
d_{\lambda/\mu}(q,t) 
&= \dfrac{1}{c_\lambda(s_k;q.t)} 
   \prod_{s \in C_{\lambda/\mu}} \dfrac{c'_\mu(s;q,t)}{c'_\lambda(s;q,t)}
   \prod_{s \in R_{\lambda/\mu}} \dfrac{c_\mu(s;q,t)}{c_\lambda(s;q,t)},
\end{align*}
where $C_{\lambda/\mu}$ and $R_{\lambda/\mu}$ were explained
at the definition \eqref{eq:psi'} of $\psi'_{\lambda/\mu}(q,t)$.
Also see \eqref{eq:cc's} for the definition of $c_\mu(s;q,t)$ and $c'_\mu(s;q,t)$.
Using the variables \eqref{eq:XY} we have
\begin{align*}
\prod_{s \in C_{\lambda/\mu}} \dfrac{c'_\mu(s;q,t)}{c'_\lambda(s;q,t)}
= \prod_{n=1}^{k-1} \dfrac{1-Y_n/X_k}{1-X_n/X_k},\qquad 
\prod_{s \in R_{\lambda/\mu}} \dfrac{c_\mu(s;q,t)}{c_\lambda(s;q,t)}
= \prod_{n=k}^{m} \dfrac{1-X_k/Y_n}{1-X_k/X_{k+1}}.
\end{align*}
We also have 
\begin{align*}
c_\lambda(s_k;q.t) = 1-t,\qquad
\prod_{n=k}^{m}\dfrac{X_{n+1}}{Y_n} = q^{-\beta_k}.
\end{align*}
Thus
\begin{align}\label{eq:qdlem}
(1-t) q^{j(s_{\lambda/\mu})-1} d_{\lambda/\mu}(q,t) 
 = \prod_{n=1}^{k-1} \dfrac{X_k-Y_n}{X_k-X_n}
   \prod_{n=k}^{m}   \dfrac{X_k-Y_k}{X_k-X_{k+1}}
 = \dfrac{ \prod_{n=1,\ n\neq k}^m (1-Y_n/X_k)}{\prod_{n=1,\ n\neq k}^{m+1}(1-X_n/X_k)}.
\end{align}

On the other hand, considering the generating function \eqref{eq:eXY} 
for $X=\{X_n \mid 1\le n \le m+1\}$ and $Y = \{Y_n \mid 1 \le n \le m \}$, 
we have
\begin{align*}
\sum_{r=0}^\infty u^r e_r[Y-X] 
&= \dfrac{\prod_{n=1}^{m}(1+ u Y_n)}{\prod_{n=1}^{m+1}(1+u X_n)}
 = \sum_{n=1}^{n+1}\dfrac{1}{1+u X_k}
    \dfrac{ \prod_{n=1,\ n\neq k}^m (1-Y_n/X_k)}{\prod_{n=1,\ n\neq k}^{m+1}(1-X_n/X_k)}.
\end{align*}
Then from \eqref{eq:qdlem} we have
\begin{align}\label{eq:qdlem2}
\sum_{r=0}^\infty u^r e_r[Y-X] 
= \sum_{n=1}^{n+1}(-u)^r B_{s_{\lambda/\mu}}^r (1-t) q^{j(s_{\lambda/\mu})-1} d_{\lambda/\mu}(q,t).
\end{align}
Now a direct calculation gives
\begin{align*}
\sum_{n=1}^{m} Y_n - \sum_{n=1}^{m+1}X_n  = (1-q)(1-t^{-1}) e_1[B_\mu]-1.
\end{align*}
Taking the coefficient of $u^r$ in \eqref{eq:qdlem2}
gives the consequence.
\end{proof}

\begin{rmk}
\begin{enumerate}
\item
In \cite[\S2]{GT} and \cite[Appendix A]{SV},
a similar formula was given.
Let us describe it in our notation.
For a fixed partition $\lambda$ we have 
\begin{align*}
(1-q)\sum_{\substack{\mu\, \vdash |\lambda|-1 \\ \mu \subset \lambda} } 
  t^{1-i(s_{\lambda/\mu})} c_{\lambda/\mu}(q,t) B_{s_{\lambda/\mu}}^k = 
 h_{k+1}[(1-q)(1-t^{-1}) B_\lambda -t/q].
\end{align*}
for $k \ge 1$.
For the cases $k=\pm1$ one can also show
\begin{align*}
\sum_{\substack{\mu\, \vdash |\lambda|-1 \\ \mu \subset \lambda} } 
 t^{1-i(s_{\lambda/\mu})} c_{\lambda/\mu}(q,t) B_{s_{\lambda/\mu}}^{\pm 1} = 
 (1-t) e_{\pm1}[B_\lambda].
\end{align*}
Here we used the symbol
$e_{-1}[X]:=\sum_i X_i^{-1}$.
The formula for $k=-1$ seems to be a new one.

\item
We also have a ``higher-order analogue" of Lemma \ref{lem:dq}.
Instead of the case $|\lambda| = |\mu| +1$,
we now consider the case $|\lambda| = |\mu| +d$ for any $d\ge 1$.
The result is given in a generating function form.
For a fixed partition $\mu$, we have 
\begin{align*}
 \sum_{\substack{\lambda\, \vdash |\mu|+d \\ \lambda \supset \mu} } 
  q^{n(\lambda')-n(\mu')} d_{\lambda/\mu}(q,t) 
  \dfrac{1}{\prod_{i=1}^d(1-u B_{s_{\lambda/\mu}^{(i)}})} 
&=\dfrac{\prod_{k=1}^{m}(1-u Y_k)}{\prod_{k=1}^{m+d}(1-u X_k)}
 \dfrac{1}{\prod_{k=1}^d(1-t^k)}
\\
&=\prod_{k=0}^{d-1} \dfrac{\widetilde{E}_\mu[u/t^k]}{1-t^{k+1}}.
\end{align*}
In the first part,
$\{ s_{\lambda/\mu}^{(i)} \mid i=1,\ldots, d \}$ is the boxes 
consisting of the skew Young diagram $\lambda/\mu$.
For defining the variables $X_k$, $Y_k$ in the second part,
let us label from left to right and from top to bottom
by $s_1,s_2,\ldots,s_{m+d}$ the attachable boxes to $\mu$ 
so that $\lambda/\mu$ forms a vertical strip.
Here $m$ is some non-negative integer.
Then define 
$$
 X_k := q^{j(s_k)-1}t^{1-i(s_k)} = B_{s_k}\ (1\le k \le m+d),\qquad
 Y_k := q^{j(s_{k+d})-1}t^{1-i(s_k)}      \ (1\le k \le m).
$$
We show an example for $\mu=(6,6,4,3,3,1,1,1)$ and $d=3$ in Figure \ref{fig:3}.

\begin{figure}[htbp]
\centering
{\unitlength 0.1in%
\begin{picture}( 14.0000, 22.0000)(  4.0000,-24.0000)%
\special{pn 8}%
\special{pa 400 800}%
\special{pa 400 2400}%
\special{fp}%
\special{pa 600 1000}%
\special{pa 600 2400}%
\special{fp}%
\special{pa 600 800}%
\special{pa 600 1000}%
\special{fp}%
\special{pa 800 1400}%
\special{pa 800 2400}%
\special{fp}%
\special{pa 1000 1400}%
\special{pa 1000 2400}%
\special{fp}%
\special{pa 1200 1800}%
\special{pa 1200 2400}%
\special{fp}%
\special{pa 1400 2000}%
\special{pa 1400 2400}%
\special{fp}%
\special{pa 1600 2000}%
\special{pa 1600 2400}%
\special{fp}%
\special{pa 400 2400}%
\special{pa 1600 2400}%
\special{fp}%
\special{pa 400 2200}%
\special{pa 1600 2200}%
\special{fp}%
\special{pa 1600 2000}%
\special{pa 400 2000}%
\special{fp}%
\special{pa 400 1800}%
\special{pa 1200 1800}%
\special{fp}%
\special{pa 1000 1600}%
\special{pa 400 1600}%
\special{fp}%
\special{pa 400 1400}%
\special{pa 1000 1400}%
\special{fp}%
\special{pa 400 1200}%
\special{pa 600 1200}%
\special{fp}%
\special{pa 600 1000}%
\special{pa 400 1000}%
\special{fp}%
\special{pa 400 800}%
\special{pa 600 800}%
\special{fp}%
\special{pn 8}%
\special{pa 400 200}%
\special{pa 590 200}%
\special{pa 590 400}%
\special{pa 400 400}%
\special{pa 400 200}%
\special{dt 0.045}%
\special{pn 8}%
\special{pa 400 400}%
\special{pa 600 400}%
\special{pa 600 600}%
\special{pa 400 600}%
\special{pa 400 400}%
\special{dt 0.045}%
\special{pn 8}%
\special{pa 400 600}%
\special{pa 600 600}%
\special{pa 600 800}%
\special{pa 400 800}%
\special{pa 400 600}%
\special{dt 0.045}%
\special{pn 8}%
\special{pa 600 800}%
\special{pa 800 800}%
\special{pa 800 1000}%
\special{pa 600 1000}%
\special{pa 600 800}%
\special{dt 0.045}%
\special{pn 8}%
\special{pa 600 1000}%
\special{pa 800 1000}%
\special{pa 800 1200}%
\special{pa 600 1200}%
\special{pa 600 1000}%
\special{dt 0.045}%
\special{pn 8}%
\special{pa 600 1200}%
\special{pa 800 1200}%
\special{pa 800 1400}%
\special{pa 600 1400}%
\special{pa 600 1200}%
\special{dt 0.045}%
\special{pn 8}%
\special{pa 1000 1400}%
\special{pa 1200 1400}%
\special{pa 1200 1600}%
\special{pa 1000 1600}%
\special{pa 1000 1400}%
\special{dt 0.045}%
\special{pn 8}%
\special{pa 1000 1600}%
\special{pa 1200 1600}%
\special{pa 1200 1800}%
\special{pa 1000 1800}%
\special{pa 1000 1600}%
\special{dt 0.045}%
\special{pn 8}%
\special{pa 1200 1800}%
\special{pa 1400 1800}%
\special{pa 1400 2000}%
\special{pa 1200 2000}%
\special{pa 1200 1800}%
\special{dt 0.045}%
\special{pn 8}%
\special{pa 1600 2000}%
\special{pa 1800 2000}%
\special{pa 1800 2200}%
\special{pa 1600 2200}%
\special{pa 1600 2000}%
\special{dt 0.045}%
\special{pn 8}%
\special{pa 1600 2200}%
\special{pa 1800 2200}%
\special{pa 1800 2400}%
\special{pa 1600 2400}%
\special{pa 1600 2200}%
\special{dt 0.045}%
\put(4.5000,-3.5000){\makebox(0,0)[lb]{$s_1$}}%
\put(4.5000,-5.5000){\makebox(0,0)[lb]{$s_2$}}%
\put(4.5000,-7.5000){\makebox(0,0)[lb]{$s_3$}}%
\put(6.5000,-9.5000){\makebox(0,0)[lb]{$s_4$}}%
\put(6.5000,-11.5000){\makebox(0,0)[lb]{$s_5$}}%
\put(6.5000,-13.5000){\makebox(0,0)[lb]{$s_6$}}%
\put(10.5000,-15.5000){\makebox(0,0)[lb]{$s_7$}}%
\put(10.5000,-17.5000){\makebox(0,0)[lb]{$s_8$}}%
\put(12.5000,-19.5000){\makebox(0,0)[lb]{$s_9$}}%
\put(16.2000,-21.5000){\makebox(0,0)[lb]{$s_{10}$}}%
\put(16.2000,-23.5000){\makebox(0,0)[lb]{$s_{11}$}}%
\end{picture}}%
\caption{Boxes $s_k$ for $\mu=(6,6,4,3,3,1,1,1)$ and $d=3$.}
\label{fig:3}
\end{figure}
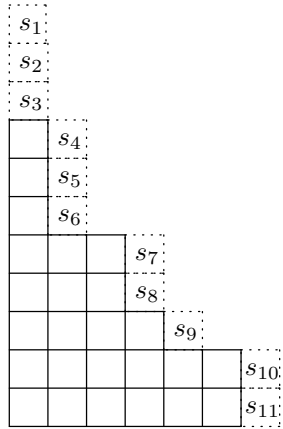

The function $\widetilde{E}_\mu[u]$ in the last part 
is defined to be 
$$
 \widetilde{E}_{\mu}[u] := \sum_{r=0}^\infty (-u)^r e_r[(1-q)(1-t^{-1}) B_\mu -1] 
 = \dfrac{1}{1-u}
   \dfrac{\prod_{s\in\lambda}(1-u B_s)(1-u B_s q/t)}{\prod_{s\in\lambda}(1-u B_s q)(1-u B_s/t)}.
$$
This formula seems to be a new one, so we mentioned it.
The proof is similar, so we omit it.
\end{enumerate}
\end{rmk}

\subsection{Part 2 of the proof}
\label{subsec:part2}

In this subsection we show the condition
\begin{align*}
D_d v_n^{(0)} = 0
\quad (d \ge2)
\end{align*}
in Corollary \ref{cor:whit:2}.
Let us recall some notations.
\begin{align}
\nonumber
v_m(0) &=  v_m^{(0)}  =
\sum_{\lambda \vdash m}  J_{\lambda}(q,t) 
 \dfrac{q^{n(\lambda')}}{\langle J_\lambda(q,t),J_\lambda(q,t) \rangle_{q,t}},
\\
\label{eq:Dz}
\quad
D(z) &= \sum_{d \ge 0}D_d z^{-d} = 
\exp\Bigl(\sum_{n>0} (1-q^{n})\partial_{p_n} z^{-n}\Bigr).
\end{align}

\begin{prop}\label{prop:v0}
For each $n \in \bb{Z}_{\ge0}$ we have 
\begin{align}\label{eq:whit:Dv}
 D_1 v_{n+1}^{(0)} = v_n^{(0)},\qquad 
 D_d v_n^{(0)} = 0\quad (d \ge 2).
\end{align}
\end{prop}

Before starting the proof, let us introduce the notation
\begin{align*}
v^{(0)}(w):=\sum_{n=0}^\infty (-w)^{n} v_n^{(0)}.
\end{align*}
Then we have  

\begin{lem}
\begin{align}\label{eq:v0}
 v^{(0)}(w) = \exp\Bigl(-\sum_{n>0} \dfrac{1}{1-q^{n}}\dfrac{p_n}{n} w^{n}\Bigr)
\end{align}
\end{lem}
\begin{proof}
Recall the Cauchy kernel \eqref{eq:Cauchy} of Macdonald symmetric functions: 
\begin{align}\label{eq:Cauchy:2}
\Pi(X,Y;q,t) := 
 \exp\Bigl(\sum_{n>0} \dfrac{1}{n}\dfrac{1-t^n}{1-q^n}p_n[X]p_n[Y]\Bigr)
 = \sum_{\lambda} \dfrac{J_\lambda[X;q,t]J_\lambda[Y;q,t]}{
    \langle J_\lambda(q,t),J_\lambda(q,t) \rangle_{q,t}}.
\end{align}
%
%
On this formula we will apply the specialization
$\widetilde{\ve}_{u,w,t}:\Lambda_{\bb{F}} \to \bb{F}[u,w]$ defined by
\begin{align*}
\widetilde{\ve}_{u,w,t} p_r := w^{r}\dfrac{u^r-1}{1-t^r} = -w^{r} \ve_{u,t} p_r .
\end{align*}
Then similarly as in \eqref{eq:spec_ve} for the specialization $\ve_{u,t}$, 
one can prove
\begin{align*}
\widetilde{\ve}_{u,w,t} J_\lambda(q,t)  = 
w^{|\lambda|} \prod_{s \in \lambda} (t^{i(s)-1} u - q^{j(s)-1})
\end{align*}

Now applying the homomorphism $\widetilde{\ve}_{0,w,t}:p_r\mapsto -w^{r}/(1-t^r)$ 
on the alphabet $Y$ 
in \eqref{eq:Cauchy} we have
\begin{align*}
 \exp\Bigl(-\sum_{n>0}\dfrac{1}{1-q^n}\dfrac{p_n}{n}w^{n}\Bigr)
 = \sum_{\lambda}(-w)^{|\lambda|} J_\lambda(q,t)
   \dfrac{\sum_{s\in\lambda}q^{j(s)-1}}{\langle J_\lambda(q,t),J_\lambda(q,t) \rangle_{q,t}}
 = v^{(0)}(w).
\end{align*}
\end{proof}

Now we turn to 
\begin{proof}[{Proof of Proposition \ref{prop:v0}}]
By the factorized formula \eqref{eq:v0} of $v^{(0)}(w)$, 
the definition \eqref{eq:Dz} of $D(z)$ 
and the Baker-Campbell-Hausdorff formula,
we have the relation of  operators acting on $\cal{F}$.
$$
 D(z) v^{(0)}(w) 
 = \exp\Bigl(-\sum_{n>0}\dfrac{1}{n}\dfrac{w^n}{z^n}\Bigr) v^{(0)}(w) D(z)
 = (1-w/z) v^{(0)}(w) D(z).  
$$
Taking the coefficient of $z^{d+1} w^{-n-1}$ with $d,n\in\bb{Z}_{\ge 0}$, 
we have
$$
D_{d+1} v^{(0)}_{n+1} =  v_{n+1}^{(0)} D_{d+1} + v_n^{(0)} D_d.
$$
Considering the action of this formula on the basis $1$ of the Fock space, we have
$$
D_1 v_{n+1}^{(0)} \cdot 1=v_n^{(0)} \cdot 1, \quad   D_d v_n^{(0)} \cdot 1=0\ (d\ge 2).
$$ 
These are the desired formulae.
\end{proof}

\begin{rmk}
\begin{enumerate}
\item
Similarly we have a ``Whittaker vector for the current $\eta(z)$".
Since 
$H[B_\lambda;u] = 1/\prod_{s\in\lambda}(1-u q^{j(s)-1}t^{1-i(s)})$,
it is natural to set 
\begin{equation}
\begin{split}
\label{eq:vinf}
v^{(\infty)}_m:=  
&\sum_{\lambda \,\vdash m}  J_{\lambda}(q,t) 
 \dfrac{q^{n(\lambda')}}{\langle J_\lambda(q,t),J_\lambda(q,t) \rangle_{q,t}}
 \dfrac{1}{\prod_{s\in\lambda} q^{j(s)-1}t^{1-i(s)} }
\\
=
&\sum_{\lambda \,\vdash m}  J_{\lambda}(q,t) 
 \dfrac{t^{n(\lambda)}}{\langle J_\lambda(q,t),J_\lambda(q,t) \rangle_{q,t}}.
\end{split}
\end{equation}
Then for each $n \in \bb{Z}_{\ge0}$ we have 
\begin{align}\label{eq:vinf_eta}
\eta_d v_{n+1}^{(\infty)} = 
 \begin{cases} v_n^{(\infty)} & (d=1) \\ 0 &(d \ge 2).\end{cases}
\end{align}

The proof of this formula is similar to that for $v^{(0)}_n$,
so we only give an outline.
Set 
$$
 v^{(\infty)}(w) := \sum_{n\ge 0}v_n^{(\infty)} w^{n} 
$$
Using the specialization $\ve_{0,t}$ of the Cauchy kernel, we can prove
$$
 v^{(\infty)} (w) = \exp\Bigl(\sum_{n>0} \dfrac{1}{1-q^{n}}\dfrac{p_n}{n} w^{n}\Bigr).
$$
Then the relation \eqref{eq:vinf_eta} follows from the following 
relation of operators acting on $\cal{F}$.
$$
  \eta(z) v^{(\infty)}(w) = (1-w/z)  v^{(\infty)} (w) \eta(z). 
$$
\item
Using the operator $\nabla$ \eqref{eq:nabla}, we have  
\begin{align}\label{eq:vinf-v0}
 v_n^{(\infty)} = \nabla^{-1} v_n^{(0)}.
\end{align} 
Then the relation \eqref{eq:vinf_eta} can be shown 
by \eqref{eq:eta-nabla} and the Whittaker relation \eqref{eq:whit:Dv} for $v_n^{(0)}$.
\end{enumerate}
\end{rmk}

\subsection{Part 3 of the proof}
\label{subsec:part3}

\subsubsection{Preparations}

Before dealing with the remaining relations 
\begin{align*}
%
&\bigl((q/t)^{d-1} \eta_d \Delta_{h_{r-1}}+ D_d \Delta_{h_{r}}\bigr) v_{n}^{(0)} = 
 0 \quad  (d \ge 2,\ r \ge 0)
\end{align*}
in Corollary \ref{cor:whit:2},
we prepare some lemmas.
Let us recall several notations.
\begin{align*}
&D(z) = \sum_{d \in \bb{Z}_{\ge0}}D_d z^{-d} = 
 \exp\Bigl(\sum_{n>0} (1-q^{n})\partial_{p_n} z^{-n}\Bigr),
\\
&\eta(z) =
 \exp\Bigl( \sum_{n>0} (1-t^{-n})\dfrac{p_n}{n}  z^{n} \Bigr)
 \exp\Bigl(-\sum_{n>0} (1-q^{n})\partial_{p_n} z^{-n}  \Bigr),
\\
\nonumber
&v^{(0)}(w) = 
  \sum_{n=0}^\infty (-w)^n v^{(0)}_n 
= \exp\Bigl(-\sum_{n>0} \dfrac{1}{1-q^{n}}\dfrac{p_n}{n} w^{-n}\Bigr).
\end{align*}

\begin{lem}\label{lem:D_eta_v}
We have the following commutation relations.
\begin{align}
\label{eq:Deta}
&D(z) \eta(w) = f(w/z)\eta(w)D(z),\quad
 f(x) := \dfrac{(1-q x)(1-x/t)}{(1-x)(1-x q/t )},  
\\
\label{eq:etaeta}
&\eta(z) \eta(w) = g(w/z)\nord{\eta(z)\eta(w)},\quad
 g(x) := \dfrac{(1-x)(1-q x/t)}{(1-q x)(1-x /t)},  
\\
\label{eq:Dv0}
& D(z) v^{(0)}(w) =  (1-w/z) v^{(0)}(w) D(z) ,
\\
\label{eq:etav0}
& \eta(z) v^{(0)}(w) =  (1-w/z)^{-1}  v^{(0)}(w) \eta(z).
\end{align}
Both sides are considered as formal series of $w/z$ 
valued in operators acting on $\cal{F}$.
\end{lem}


\begin{lem}\label{lem:expand} 
Consider $(1-w/z)^{-1}$ is considered as a formal series in $w/z$.
\begin{enumerate}
\item 
For a formal series $a(z)= \sum_{d\in \bb{Z} }a_{-d}z^d$,
the coefficient of  $z^0$ in $(1-w/z)^{-1} a(z)$ is 
$\sum_{d\ge0}a_{-d}w^d$.
\item
For a formal series $a_{-}(z)= \sum_{d\ge0}a_{-d}z^d$,
we have 
$$
 (1-w/z)^{-1} a_{-}(z) = (1-w/z)^{-1}a_{-}(w) + O(z).
$$
\item
For a formal series $a_{+}(z)= \sum_{d\ge0}a_{+d}z^{-d}$,
we have 
$$
 (1-w/z)^{-1} a_{+}(z) = a_0 + O(z).
$$
\end{enumerate}
\end{lem}

The proofs are straightforward, so we omit them.

\subsubsection{The case $r=1$}\label{sssec:r=1}

As a demonstration, let us show 
\begin{align}\label{eq:r=1}
&\bigl((q/t)^{d-1} \eta_d + D_d \Delta_{h_{1}}\bigr) v_{n}^{(0)} = 
 0 \quad  (d \ge 2)
\end{align}
Hereafter we will use  
$$
 \eta_{+}(z) := \sum_{d\ge0}\eta_d z^{-d},\qquad
 \eta_{-}(z) := \sum_{d\le0}\eta_d z^{-d}.
$$
We will also denote by the symbol $\equiv$ the equivalence of formal series 
up to $z^{-2}$ (ignoring $z^{-1},z^0,z^1,\ldots$ parts).
Thus the relation\eqref{eq:r=1} can be expressed as
\begin{align}\label{eq:r=1:2}
 \bigl(t/q \cdot \eta_{+}( z t/q) + D(z) \Delta_{h_{1}}\bigr) v^{(0)}(w) \equiv 0.
\end{align}

By \eqref{eq:Dv0} in Lemma \ref{lem:D_eta_v} we have
\begin{align*}
D(z) \eta(z_1) v^{(0)} (w) 
 = (1- w/z) (1-z_1/w)^{-1} f(z_1/z) v^{(0)}(w) \eta(z_1) D(z).
\end{align*}
Thus as formal series in $\cal{F}$ we have
\begin{align*}
D(z) \eta(z_1) v^{(0)} (w) \cdot 1_{\cal{F}}
 = (1- w/z) (1-z_1/w)^{-1} f(z_1/z) v^{(0)}(w) \eta_{-}(z_1) \cdot 1_{\cal{F}}.
\end{align*}
Taking coefficients of ${z_1}^0$ in both sides and using Lemma \ref{lem:expand}, we have
\begin{align*}
D(z) \, \eta_0  \, v^{(0)} (w) \cdot 1_{\cal{F}}
&= (1- w/z) f(w/z) \cdot v^{(0)}(w) \eta_{-}(w)  \cdot 1_{\cal{F}}
\\
&= \Bigl(-\dfrac{w}{z} + (q+t-1) t/q + \dfrac{(1-q)(1-t^{-1}) t /q}{1-q/t \cdot w/z} \Bigr) \cdot 
   v^{(0)}(w) \eta_{-}(w)  \cdot 1_{\cal{F}}.
\\
&\equiv (1-q)(1-t^{-1}) t /q \cdot (1-q/t \cdot w/z)^{-1} \cdot 
   v^{(0)}(w) \eta_{-}(w)  \cdot 1_{\cal{F}}.
\end{align*}
In the second line we used the formula
$$
(1-x) f(x ) 
 = -x + (q+t-1) t/q + \dfrac{(1-q)(1-t^{-1}) t/q}{1-x q/t}.
$$
Now recall that by \eqref{eq:eta0Delta} 
we have $\Delta_{h_1} = \Delta_{e_1} = (1-\eta_0)/\bigl((1-q) (1-t^{-1})\bigr)$.
Using \eqref{eq:Dv0} we have
\begin{align}
\label{eq:r=1:D}
D(z) \Delta_{h_1} v^{(0)} (w) \cdot 1_{\cal{F}}
 \equiv - t /q \cdot (1-q/t \cdot w/z)^{-1} \cdot 
   v^{(0)}(w) \eta_{-}(w)  \cdot 1_{\cal{F}}.
\end{align}

On the other hand, by \eqref{eq:etav0} in Lemma \ref{lem:D_eta_v} we have
\begin{align}
\nonumber
\eta_{+}(z t/q ) v^{(0)} (w) \cdot 1_{\cal{F}}
&= (1- q/t \cdot w/z)^{-1} v^{(0)}(w) \eta_{-}(z t/q)  \cdot 1_{\cal{F}}
\\
\label{eq:r=1:eta}
&= (1- q/t \cdot w/z)^{-1} v^{(0)}(w) \eta_{-}(w)  \cdot 1_{\cal{F}}.
\end{align}
In the last line we used Lemma \ref{lem:expand}.
Comparing \eqref{eq:r=1:D} and \eqref{eq:r=1:eta},
we have the desired formula \eqref{eq:r=1:2}.

\subsubsection{Calculation of $D(z) \Delta_{h_r} v^{0}(w)$}

We want to calculate $D(z) \Delta_{h_r} v^{0}(w)$,
The difficult point is that we don't have an explicit formula for $\Delta_{h_r}$.
However we may use the explicit formula \eqref{eq:E_r} for $\widehat{E}_r$ 
in Fact \ref{fct:E_r} and the translation formula 
between $e_r[s^\lambda]$ and $h_r[B_\lambda]$ in Lemma \ref{lem:B-s}.
Thus we start with $D(z) \widehat{E}_r v^{0}(w)$.

\begin{lem}\label{lem:DErv0}
For each $r \in \bb{Z}_{\ge 1}$ we have 
\begin{align*}
D(z) \widehat{E}_r v^{0}(w) \cdot 1_\cal{F}
= \dfrac{t^{-r(r+1)/2}}{(t^{-1};t^{-1})_r} 
  (1-w/z)
  \prod_{i=1}^{r-1} f(t^{-i} w/z) \cdot  
  v^{(0)}(w) \prod_{i=1}^{r-1} \eta_{-}(w/t^{i-1}) \cdot 1_\cal{F}
\end{align*}
as a Laurent formal series of $w/z$ valued in $\cal{F}$.
\end{lem}

\begin{proof}
By Fact \ref{fct:E_r} and Lemma \ref{lem:D_eta_v},
we have  
\begin{align*}
&D(z) \widehat{E}_r v^{(0)}(w) \cdot 1_\cal{F}
= \dfrac{t^{-r(r+1)/2}}{(t^{-1};t^{-1})_r} 
  D(z) 
  \Bigl[\prod_{1\le i<j \le r}\ve(z_j/z_i) \cdot 
        \nord{\prod_{i=1}^{r} \eta(z_i)} \Bigr]_1 
  v^{(0)}(w) \cdot 1_\cal{F}
\\
&= \dfrac{t^{-r(r+1)/2}}{(t^{-1};t^{-1})_r} 
  \Bigl[
   (1-w/z)\prod_{i=1}^r \dfrac{f(z_i/z)}{1-w/z_i}
   \prod_{1\le i<j \le r}\ve(z_j/z_i) \cdot 
   v^{(0)}(w) \nord{\prod_{i=1}^{r} \eta(z_i)} D(z) 
\Bigr]_1 \cdot 1_\cal{F}.
\\
&= \dfrac{t^{-r(r+1)/2}}{(t^{-1};t^{-1})_r} 
  \Bigl[
   (1-w/z)\prod_{i=1}^r \dfrac{f(z_i/z) }{1-w/z_i}
   \prod_{1\le i<j \le r}\ve(z_j/z_i) \cdot 
   v^{(0)}(w) \prod_{i=1}^{r} \eta_{-}(z_i) 
\Bigr]_1 \cdot 1_\cal{F}.
\end{align*}
Now we want to take the coefficient of ${z_r}^0$.
By Lemma \ref{lem:expand},
we need only to consider the residue at $z_r = w$.
Thus we have 
\begin{align*}
&D(z) \widehat{E}_r v^{(0)}(w) \cdot 1_{\cal{F}}
\\
&= \dfrac{t^{-r(r+1)/2}}{(t^{-1};t^{-1})_r} 
  \Bigl[
   (1-w/z)f(w/z) \prod_{i=1}^{r-1} \dfrac{ f(z_i/z) \ve(w/z_i)  }{ 1-w/z_i }
   \prod_{1\le i<j \le r-1}\ve(z_j/z_i) 
   \cdot v^{(0)}(w) \eta_{-}(w) \prod_{i=1}^{r-1} \eta_{-}(z_i)
\Bigr]_1 \cdot 1_{\cal{F}}
\\
&=\dfrac{t^{-r(r+1)/2}}{(t^{-1};t^{-1})_r} 
  \Bigl[
   (1-w/z)f(w/z) 
   \prod_{i=1}^{r-2} \dfrac{f(z_i/z) \ve(w/z_i) \ve(z_{r-1}/z_i)}{1-w/z_i} 
   \prod_{1\le i<j \le r-2}\ve(z_j/z_i) \cdot 
\\
&\phantom{= \dfrac{t^{-r(r+1)/2}}{(t^{-1};t^{-1})_r} \Bigl[(1} 
\cdot \dfrac{ \ve(w/z_{r-1}) f(z_{r-1}/w)}{1-w/z_{r-1}}  \cdot 
 v^{(0)}(w) \eta_{-}(w) \prod_{i=1}^{r-1} \eta_{-}(z_i) 
\Bigr]_1 \cdot 1_{\cal{F}}
\\
&=\dfrac{t^{-r(r+1)/2}}{(t^{-1};t^{-1})_r} 
  \Bigl[
   (1-w/z)f(w/z) 
   \prod_{i=1}^{r-2} \dfrac{ f(z_i/z) \ve(z_{r-1}/z_i) }{ 1-t^{-1} w/z_i }
   \prod_{1\le i<j \le r-2}\ve(z_j/z_i) \cdot 
\\
&\phantom{= \dfrac{t^{-r(r+1)/2}}{(t^{-1};t^{-1})_r} \Bigl[(1} 
\cdot \dfrac{ f(z_{r-1}/z) }{ 1-t^{-1} w/z_{r-1} }  \cdot 
 v^{(0)}(w) \eta_{-}(w) \prod_{i=1}^{r-1} \eta_{-}(z_i)
\Bigr]_1 \cdot 1_\cal{F}.
\end{align*}
Then we see that the coefficient of ${z_{r-1}}^0$,
is given by the residue at $z_{r-1} = w/t$.
Now we have
\begin{align*}
&D(z) \widehat{E}_r v^{(0)}(w) \cdot 1_\cal{F}
\\
&=\dfrac{t^{-r(r+1)/2}}{(t^{-1};t^{-1})_r} 
  \Bigl[
   (1-w/z)f(w/z) f(t^{-1} w/z) 
   \prod_{i=1}^{r-2} \dfrac{ f(z_i/z)  \ve(w/ t z_i)}{ 1-t^{-1} w/z_i }   
   \prod_{1\le i<j \le r-2}\ve(z_j/z_i) \cdot 
\\
&\phantom{= \dfrac{t^{-r(r+1)/2}}{(t^{-1};t^{-1})_r} \Bigl[(1} 
  \cdot 
 v^{(0)}(w) \eta_{-}(w) \eta_{-}(w/t) \prod_{i=1}^{r-2} \eta_{-}(z_i) 
\Bigr]_1 \cdot 1_\cal{F}
\\
&=\dfrac{t^{-r(r+1)/2}}{(t^{-1};t^{-1})_r} 
  \Bigl[
   (1-w/z)f(w/z) f(t^{-1} w/z) 
   \prod_{i=1}^{r-2} \dfrac{ f(z_i/z)  }{ 1-t^{-2} w/z_i }
   \prod_{1\le i<j \le r-2}\ve(z_j/z_i) \cdot 
\\
&\phantom{= \dfrac{t^{-r(r+1)/2}}{(t^{-1};t^{-1})_r} \Bigl[(1} 
  \cdot 
 v^{(0)}(w) \eta_{-}(w) \eta_{-}(w/t) \prod_{i=1}^{r-2} \eta_{-}(z_i)
\Bigr]_1 \cdot 1_\cal{F}.
\end{align*}
Repeating this process,
we see that it is enough to consider the residue 
at $z_i = w/t^{r-i}$  ($1 \le i \le r$),
and we get the consequence immediately.
\end{proof}

Next we want to compute $D \widehat{E}_{\lambda} v_0(w)$
with
$$
 \widehat{E}_\lambda :=  \widehat{E}_{\lambda_1} \widehat{E}_{\lambda_2} \cdots
 \widehat{E}_{\lambda_\ell}
$$ 
for an arbitrary  partition $\lambda$.
We need the following notions.

%

\begin{dfn}\label{dfn:wmu}
Let $\mu$ be a partition with length $\ell$.
\begin{enumerate}
\item
We define the set of variables $w^\mu = \{ w^{\mu}_s \mid s \in \mu \}$ by
$$
 w^{\mu}_s := w q^{i(s)-1} t^{l(s)}.
$$
\item
We set
\begin{align*}
\eta_{-}(w^{\mu}) := \prod_{s \in \mu} \eta_{-}(w^{\mu}_s),\qquad
f(w^{\mu}/z) := \prod_{s \in \mu} f(w^{\mu}_s/z).
\end{align*}
\end{enumerate}
\end{dfn}

If we need to specify the order in the variables $w^{\mu}$,
the boxes in the diagram of $\mu$ are counted from left to right and from top to bottom.
So we have 
\begin{align}\label{eq:wlam_ord}
 w^{\mu} 
  = (q^{\ell-1} w/t^{\mu_\ell-1},q^{\ell-1} w/t^{\mu_\ell-2},\ldots,q^{\ell-1} w, \ldots,
     q w/t^{\mu_2-1},\ldots, q w,w/t^{\mu_1-1},\ldots,w)
\end{align}
with $\ell := \ell(\mu)$.

\begin{prop}\label{prop:DElamv0}
For a partition $\lambda$ we have 
\begin{align*}
 D(z) \widehat{E}_{\lambda} v_0(w) \cdot 1_\cal{F}
= C_{\lambda}(t)
  \sum_{\mu \, \vdash |\lambda|} 
  (1-w/z) f(w^\mu/z) R_{\lambda,\mu}(q,t) \cdot v^{(0)}(w) \eta_{-}(w^\mu) 
  \cdot 1_\cal{F}
\end{align*}
with 
\begin{align}
\label{eq:Cl}
&C_{\lambda}(t) := 
\prod_{i=1}^{\ell(\lambda)}
  \dfrac{t^{-\lambda_i(\lambda_i+1)/2}}{(t^{-1};t^{-1})_{\lambda_i}},  
\end{align}
and
\begin{align}
\label{eq:Rlm}
&R_{\lambda,\mu}(q,t) := 
\Res_{(z_i) = w^\mu}\Bigl[
 \prod_{i=1}^{\ell(\lambda)} \prod_{1\le j<k \le \lambda_i} 
  \ve(z_{\lambda^{(i-1)}+k}/z_{\lambda^{(i-1)}+j})
 \prod_{1\le i<j\le \ell(\lambda)}\prod_{k=1}^{\lambda_i}\prod_{l=1}^{\lambda_j}
  g(z_{\lambda^{(j-1)}+l}/z_{\lambda^{(i-1)}+k}) \Bigr].
\end{align}
When taking the residue, the set $w^\lambda$ is ordered as \eqref{eq:wlam_ord}.
Also we used the symbol 
\begin{align}\label{eq:lambra}
\lambda^{(i)} := \sum_{k=1}^{i}\lambda_k \quad (1\le i \le \ell(\lambda)),\qquad
\lambda^{(0)} := 0.
\end{align}
\end{prop}

\begin{proof}
As in the proof of Lemma \ref{lem:DErv0},
we have 
\begin{align*}
& D(z) \widehat{E}_{\lambda} v_0(w) \cdot 1_\cal{F}
\\
&=\prod_{i=1}^{\ell(\lambda)}
  \dfrac{t^{-\lambda_i(\lambda_i+1)/2}}{(t^{-1};t^{-1})_{\lambda_i}} \cdot 
 (1-w/z)  \Bigl[
   \prod_{i=1}^{|\lambda|} \dfrac{f(z_i/z)}{ 1-w/z_i } 
   \prod_{i=1}^{\ell(\lambda)} \prod_{1\le j<k \le \lambda_i} 
    \ve(z_{\lambda^{(i-1)}+k}/z_{\lambda^{(i-1)}+j}) \cdot
\\
&\phantom{\prod_{i=1}^{\ell(\lambda)}
  \dfrac{t^{-\lambda_i(\lambda_i+1)/2}}{(t^{-1};t^{-1})_{\lambda_i}} \cdot 
 (1-w/z) \cdot \Bigl[\prod }
   \cdot \prod_{1\le i<j\le \ell(\lambda)}\prod_{k=1}^{\lambda_i}\prod_{l=1}^{\lambda_j}
    g(z_{\lambda^{(j-1)}+l}/z_{\lambda^{(i-1)}+k}) 
  \cdot v^{(0)}(w) \prod_{i=1}^{|\lambda|} \eta_{-}(z_i) 
  \Bigr]_1 \cdot 1_\cal{F}.
\end{align*}
One can check that the poles of the factor in the bracket $[\,\cdots ]_1$ 
are only at $(z_i) = w^{\mu}$ with $|\mu| = |\lambda|$.
Thus we have the consequence.
\end{proof}

For later use, we remark 
\begin{lem}\label{lem:fwpoles}
For a parition $\mu$,
the poles of $(1-w) \prod_{s \in \mu} f(w^{\mu}_s)$ are of order one 
and placed at
$$
 \{ w= t/q \cdot B_{s(\mu'/\nu')} \mid \nu \vdash |\mu|-1, \ \nu \subset \mu \}.
$$
Here $s(\lambda/\mu)$ denotes the (unique) box of the skew Young diagram $\lambda/\mu$.
The symbol $B_s$ for a box $s$ is defined at \eqref{eq:Bs}.
\end{lem}

The proof is by an elementary calculation, so we omit it.

\subsubsection{Calculation of $\eta_+(z) \Delta_{h_r} v^{0}(w)$}

As in the previous part,
we start with the calculation of $\eta_+(z) \widehat{E}_r v^{0}(w) \cdot 1_\cal{F}$.

\begin{lem}
For each $r \in \bb{Z}_{\ge 1}$ we have 
\begin{align*}
\eta_{+}(z) \widehat{E}_r v^{(0)}(w) \cdot 1_\cal{F}
=\dfrac{t^{-r(r+1)/2}}{(t^{-1};t^{-1})_r} 
   (1-w/z)^{-1} \prod_{i=0}^{r-1} g(t^{-i} w/z) \cdot 
   v^{(0)}(w) \eta_{-}(z) \eta_{-}(w^{(r-1)}) \cdot 1_\cal{F}.
\end{align*}
as a Laurent formal series of $w/z$ valued in $\cal{F}$.
%
%
\end{lem}

\begin{proof}
The proof is similar to that of Lemma \ref{lem:DErv0}.
By Fact \ref{fct:E_r} and Lemma \ref{lem:D_eta_v},
we have 
\begin{align*}
&\eta(z) \widehat{E}_r v^{(0)}(w)  
\\
&= \dfrac{t^{-r(r+1)/2}}{(t^{-1};t^{-1})_r} 
  \Bigl[
   (1-w/z)^{-1} \prod_{i=1}^r \dfrac{ g(z_i/z) }{ 1-w/z_i }
   \prod_{1\le i<j \le r}\ve(z_j/z_i) \cdot 
   v^{(0)}(w) \eta_{-}(z) \prod_{i=1}^{r} \eta_{-}(z_i)  
\Bigr]_1.
\end{align*}
Now consider the coefficient of ${z_r}^0$.
By Lemma \ref{lem:expand},
we need only to consider the residue at $z_r = w$.
Thus
\begin{align*}
&\eta_{+}(z) \widehat{E}_r v^{(0)}(w) \cdot 1_\cal{F}
\\
&= \dfrac{t^{-r(r+1)/2}}{(t^{-1};t^{-1})_r} 
  \Bigl[
   \dfrac{g(w/z)}{ 1-w/z } 
   \prod_{i=1}^{r-1} \dfrac{ g(z_i/z) \ve(w/z_i) }{ 1-w/z_i }  
   \prod_{1\le i<j \le r-1}\ve(z_j/z_i)  
   \cdot v^{(0)}(w) \eta_{-}(z) \eta_{-}(w)  \prod_{i=1}^{r-1} \eta_{-}(z_i)  
 \Bigr]_1 \cdot 1_\cal{F}
\\
&=\dfrac{t^{-r(r+1)/2}}{(t^{-1};t^{-1})_r} 
  \Bigl[
   \dfrac{g(w/z)}{ 1-w/z } 
   \prod_{i=1}^{r-1} \dfrac{g(z_i/z)}{ 1-t^{-1} w/z_i }   
   \prod_{1\le i<j \le r-1}\ve(z_j/z_i)  
   \cdot v^{(0)}(w) \eta_{-}(z) \eta_{-}(w)  \prod_{i=1}^{r-1} \eta_{-}(z_i)  
  \Bigr]_1 \cdot 1_\cal{F}.
\end{align*}
Then we see that the coefficient of ${z_{r-1}}^0$,
is given by the residue at $z_{r-1} = w/t$.
Repeating this process,
we see that it is enough to consider the residue 
at $z_i = w/t^{r-i}$ ($1 \le i \le r$).
Then one can get the result immediately.
\end{proof}

As in Definition \ref{dfn:wmu}, we set
$$
 g(w^{\mu}/z) := \prod_{s \in \mu} g(w^{\mu}_s/z).
$$

\begin{lem}\label{lem:etaElamv0}
For a partition $\lambda$ we have 
\begin{align*}
 \eta_{+}(z) \widehat{E}_{\lambda} v_0(w) \cdot 1_\cal{F}
= C_{\lambda}(t)
  \sum_{\mu \, \vdash |\lambda|} 
  (1-w/z)^{-1}  g(w^{\mu}/z)
  R_{\lambda,\mu}(q,t) \cdot v^{(0)}(w) \eta_{-}(z)\eta_{-}(w^\mu) 
  \cdot 1_\cal{F},
\end{align*}
where $C_{\lambda}(t)$ is given by \eqref{eq:Cl}
and $R_{\lambda,\mu}(q,t)$ is given by \eqref{eq:Rlm}.
\end{lem}

\begin{proof}
The proof is similar as in Proposition \ref{prop:DElamv0}, so we omit it.
\end{proof}

By an elementary calculation, we also have
\begin{lem}\label{lem:gwpoles}
The poles of $(1-w)^{-1} \prod_{s \in w_\mu}g(w^{\mu}_s)$ are of order one,
and are placed at 
$$
 \{w = {B_{s(\lambda'/\mu')}}^{-1} \mid \lambda \vdash |\mu| + 1, \lambda \supset \mu \}.
$$
\end{lem}

\begin{rmk}
\begin{enumerate}
\item
The boxes appearing in the description of the poles 
are the attachable boxes to $\mu$,
which also appeared in Garsia and Tesler's change of variables in Lemma \ref{lem:dq}.
\item
The transposed diagrams appear in the description,
which seems to be related to the similar phenomena in the description of 
Gordon filtration \cite{FHHSY}
on the space of bosonized Macdonald difference operators.
\end{enumerate}
\end{rmk}

Then by Lemma \ref{lem:expand} (3), Lemma \ref{prop:etaElamv0} and Lemma \ref{lem:gwpoles}
we conclude 
\begin{prop}\label{prop:etaElamv0}
For a partition $\lambda$ we have 
\begin{align*}
\eta_{+}(z) \widehat{E}_{\lambda} v_0(w) \cdot 1_\cal{F}
= C_{\lambda}(t)
  \sum_{\mu \, \vdash |\lambda|} 
  R_{\lambda,\mu}(q,t)
  \sum_{\nu \, \vdash |\mu|+1,\ \nu \supset \mu} 
  \dfrac{G_{\nu,\mu}(q,t)}{1-w^\nu_{s(\nu/\mu)}/z}
  \cdot v^{(0)}(w) \eta_{-}(w^\nu) 
  \cdot 1_\cal{F}
\end{align*}
with $G_{\nu,\mu}(q,t) := \Res_{w=z B_{s(\nu'/\mu')}} (1-w/z )^{-1} g(w^{\mu}/z)$.
\end{prop}

\subsubsection{The remaining part of the proof}

Let us recall the relation we must show.
\begin{align}\label{eq:main:r}
 \bigl(t/q \cdot \eta_{+}( z t/q) \Delta_{h_{r-1}} + D(z) \Delta_{h_{r}}\bigr) 
 v^{(0)}(w) \cdot 1_{\cal{F}} \equiv 0.
\end{align}
Here we used the symbol $\equiv$ as in \S\ref{sssec:r=1}.

By Lemma \ref{lem:B-s} we have
\begin{align*}
 \Delta_{h_r}
&= \sum_{k=0}^r (-t)^k \theta_{r-k} 
   \sum_{r_i \ge 1,\, \sum_i r_i = k} 
    q^{\sum_{i}(i-1)r_i} \widehat{E}_{r_1}\widehat{E}_{r_2}\cdots 
\\
&= \sum_{k=0}^r (-t)^k \theta_{r-k} 
   \sum_{\lambda \,\vdash k} \widehat{E}_\lambda
    \sum_{\{r_i\} = \lambda }     q^{\sum_{i}(i-1)r_i}. 
\end{align*}
At the last line that the sequence $(r_i)$ runs over the set of permutations of $\lambda$.
We write this formula as 
\begin{align}\label{eq:DHrwH}
 \Delta_{h_r}
= \sum_{|\lambda| \le r}  \delta_\lambda(q,t) \widehat{E}_\lambda,\qquad
 \delta_\lambda(q,t) 
:=(-t)^{|\lambda|} \theta_{r-|\lambda|} \sum_{\{r_i\} = \lambda }     q^{\sum_{i}(i-1)r_i}.
\end{align}

By Proposition \ref{prop:DElamv0}, Lemma \ref{lem:fwpoles} and \eqref{eq:DHrwH}, 
we have 
\begin{align*}
&D(z) \Delta_{h_{r}} v^{(0)}(w) \cdot 1_{\cal{F}}
\\
&\equiv
 \sum_{|\lambda| \le r} 
  \delta_\lambda(q,t) C_{\lambda}(t)
  \sum_{ \mu\, \vdash |\lambda| } 
  R_{\lambda,\mu}(q,t) 
  \sum_{\nu \, \vdash |\mu|-1, \ \nu \subset \mu }
  \dfrac{F_{\nu,\mu}(q,t)}{1-q/t \cdot w^{\mu}_{s(\mu/\nu)} /z}
  \cdot v_0(w) \eta_{-}(w^\mu) \cdot 1_{\cal{F}}
\\
&=
 \sum_{|\mu| \le r} v_0(w) \eta_{-}(w^\mu) \cdot 1_{\cal{F}} 
  \sum_{\nu \, \vdash |\mu|-1, \ \nu \subset \mu }
   \dfrac{F_{\nu,\mu}(q,t)}{1-q/t \cdot w^{\mu}_{s(\mu/\nu)} /z}
  \sum_{ \lambda\, \vdash |\mu| } 
   \delta_\lambda(q,t) C_{\lambda}(t) R_{\lambda,\mu}(q,t) 
\end{align*}
with
$F_{\nu,\mu}(q,t) := \Res_{w=z B_{s(\mu'/\nu')}} (1-w/z )^{-1} g(w^{\mu}/z)$.

On the other hand, 
by Proposition \ref{prop:etaElamv0}, Lemma \ref{lem:gwpoles} and \eqref{eq:DHrwH}, 
we see that $\eta_{+}(z) \Delta_{h_{r-1}} v^{(0)}(w)  \cdot 1_{\cal{F}}$ is of the form
\begin{align*}
&\eta_{+}(z) \Delta_{h_{r-1}} v^{(0)}(w) \cdot 1_{\cal{F}}
\\
&=
 \sum_{|\lambda| \le r-1} 
  \delta_\lambda(q,t) C_{\lambda}(t)
  \sum_{\mu \, \vdash |\lambda|} 
  R_{\lambda,\mu}(q,t)
  \sum_{\nu \, \vdash |\mu|+1,\ \nu \supset \mu} 
  \dfrac{G_{\nu,\mu}(q,t)}{1-w^\nu_{s(\nu/\mu)}/z}
  \cdot v^{(0)}(w) \eta_{-}(w^\nu) 
  \cdot 1_\cal{F}
\\
&=
 \sum_{|\mu| \le r} v^{(0)}(w) \eta_{-}(w^\mu) \cdot 1_\cal{F}
  \sum_{\nu \, \vdash |\mu|-1,\ \nu \subset \mu} 
  \dfrac{G_{\mu,\nu}(q,t)}{1-w^\mu_{s(\mu/\nu)}/z}
 \sum_{ \lambda \, \vdash |\mu| -1} 
  \delta_\lambda(q,t) C_{\lambda}(t) R_{\lambda,\mu}(q,t).
\end{align*}
Now by a direct calculation one can show 
\begin{align*}
 F_{\nu,\mu}(q,t) 
 \sum_{ \lambda\, \vdash |\mu|} 
  \delta_\lambda(q,t) C_{\lambda}(t) R_{\lambda,\mu}(q,t) 
=
 t/q \cdot G_{\mu,\nu}(q,t) 
 \sum_{ \lambda \, \vdash |\mu| -1} 
  \delta_\lambda(q,t) C_{\lambda}(t) R_{\lambda,\mu}(q,t),
\end{align*}
which implies the desired consequence \eqref{eq:main:r}.

\section{Interpretation via geometry of Hilbert scheme of points over plane}
\label{sec:geom}

Some parts of our result can be interpreted 
in terms of the geometry of Hilbert scheme of points over the affine plane.

\subsection{Recollection of Haiman's work}
\label{ssec:Haiman}

For the explanation, let us recall Haiman's work \cite{H:2001,H:2002,H:2003} briefly.
Our presentation follows \cite[\S1]{N:2012}.

Set $X := \bb{C}^2$.
Denote by $X^{[n]}$ the Hilbert scheme of $n$ points in $X$.
Set theoretically it is the collection of ideals $I$ in $\bb{C}[x,y]$ 
with $\dim_{\bb{C}} \bb{C}[x,y]/I = n$.

$X$ has a natural action of the torus $\bb{T} := \bb{C}^* \times  \bb{C}^*$,
and $X^{[n]}$ has an induced action.
Let us express the $\bb{T}$-action on $X$ by $(x,y) \mapsto (t x, q y)$.

What we want to mention is the following famous result due to Haiman \cite{H:2003}:
The Macdonald symmetric functions correspond to 
elements in the equivariant $K$-group $K_{\bb{T}}(X^{[n]})$
given by $\bb{T}$-fixed points.
Note that $K_{\bb{T}}(X^{[n]})$ is a module over 
the representation ring $R(\bb{T}) \cong \bb{Z}[q^{\pm1},t^{\pm1}]$.
Let us describe this fact in detail.

The $\bb{T}$-fixed points in $X^{[n]}$ correspond to monomial ideals in $\bb{C}[x,y]$, 
so that they are parametrized by partitions $\lambda$ of $n$. 
We denote by $I_{\lambda} \in X^{[n]}$ the corresponding $\bb{T}$-fixed point.

The Hilbert scheme $X^{[n]}$ is a fine moduli scheme 
and has a tautological bundle $\cal{B}$.
It is a vector bundle whose fiber at the point corresponding to the ideal $I$ in $\bb{C}[x,y]$
is given by $\bb{C}[x,y]/I$.
The fiber over the fixed point $I_\lambda$ is a $\bb{T}$-module, 
and its character is given by 
\begin{align}\label{eq:chB}
 \ch \cal{B}_{I_\lambda} = \sum_{s\in\lambda}q^{j(s)-1}t^{i(s)-1} = e_1[B_\lambda(q,t^{-1})]
\end{align}
Here the alphabet $B_\lambda$ \eqref{eq:Bs} naturally appears.

We give another formula of $\bb{T}$-character.
Let us denote by $\wedge_{-1} T^*X^{[n]}$
The alternating sum of exterior powers of theT�@cotangent bundle $T^* X^{[n]}$.
Then its fiber on the fixed point $I_\lambda$ has the following $\bb{T}$-character.
\begin{align}\label{eq:ch:T}
\ch \wedge_{-1} T_{I_\lambda}^*X^{[n]}
=\sum_{s \in \lambda}
 (1 -t^{-l_\lambda(s)} q^{a_\lambda(s)+1})(1 -t^{l_\lambda(s)+1} q^{-a_\lambda(s)})
=c_{\lambda}(q^{-1},t)c'_{\lambda}(q,t^{-1}).
\end{align}

The main object in \cite{H:2001} is the isospectral Hilbert scheme $X_n$,
which is defined to be the reduced fiber product
\begin{align*}
\xymatrix{
 X_n \ar[r]^{f} \ar[d]_{\rho} \ar@{}[dr]|\square
 & X^n \ar[d]
\\
 X^{[n]} \ar[r]  & S^n X,
}
\end{align*}
where $S^n X = X^n/S_n$ is the symmetric product of $X$.
The main result in \cite{H:2001} states that 
$X_n$ is normal, Cohen-Macaulay and Gorenstein.

As an application of this result together with the result by Bridgeland-King-Reid \cite{BKR}, 
Haiman \cite{H:2002} proved that the functor
$$ 
 \mathbf{R}f_* \circ \rho^{*} :  D^b(\Coh X^{[n]}) \longrightarrow D^b_{S_n}(\Coh X^n)
$$
is an equivalence of categories.
Here $D^b(-)$ denotes the bounded derived category of coherent sheaves,
and  $D^b_{S_n}(-)$ is the derived category of $S_n$-equivariant coherent sheaves.
This equivalence holds also for $\bb{T}$-equivariant derived categories.

The category $D^b_{S_n}(\Coh X^n)$ is identified with 
the derived category of bounded complexes of finitely generated 
$S_n$-equivariant $\bb{C}[x, y]^{\otimes n}$-modules.

On the equivariant Grothendieck group level, we have a natural isomorphism
\begin{align}\label{eq:KH}
 K_{\bb{T}}(X^{[n]}) \xrightarrow{\ \sim \ }
 K_{S_n \times \bb{T}}(X^n).
\end{align}

In the present situation, 
one can consider the push-forward homomorphism for $X_n \to \mathrm{pt}$. 
Although this morphism is not a proper,
the push-forward map is well-defined on the localized Grothendieck group
$K_{S_n \times \bb{T}}(\mathrm{pt}) \otimes_{R(\bb{T})} \mathrm{Frac}(R(\bb{T}))
 \cong 
 R(S_n \times \bb{T}) \otimes_{\bb{Z}[q^{\pm 1},t^{\pm 1}]} \bb{Q}(q,t)
 \cong
 R(S_n)  \otimes \bb{F}$
by the Atiyah-Bott-Lefschetz localization formula. 
($\mathrm{Frac}$ denotes the field of fractions.)
Recall also that via the  Frobenius map 
the representation ring $R(S_n)$ is isomorphic to $\Lambda_n$,
the degree $n$ part of the ring $\Lambda$ of symmetric functions.
Combining the push-forward homomorphism and the Frobenius map,
we have $K_{S_n \times \bb{T}}(X^n) \to \Lambda_{\bb{F},n}$.

Composing the last map with the isomorphism \eqref{eq:KH} we have
$$
 \Phi: K_{\bb{T}}(X^{[n]}) \to \Lambda_{\bb{F},n}.
$$
The image of $\Phi$ is described in \cite[Proposition 5.4.6]{H:2003}:
$\Phi$ induces an isomorphism
$$
 K_{\bb{T}}(X^{[n]}) \xrightarrow{\ \sim \ } 
 \bigl\{ f[Y] \in \Lambda_{\bb{F},n} 
       \mid f[(1-q)(1-t)Y] \in \Lambda_{\bb{Z}[q^{\pm1}, t^{\pm1}],n}  
 \bigr\}.
$$
Here we used the plethystic notation.

Now we can state the result of Haiman \cite{H:2001}.
For a partition $\lambda$ of $n$,
the image of the $\bb{T}$-fixed point $I_\lambda$ is given by 
$$
 \Phi([I_\lambda]) = \widetilde{H}_\lambda,
$$
where $\widetilde{H}_\lambda$ is the modified Macdonald symmetric function 
defined to be
\begin{align}\label{eq:HaimanH}
 \widetilde{H}_\lambda[Y;q,t] 
  := t^{n(\lambda)}J_\lambda[Y/(1-t^{-1});q,t^{-1}].
\end{align}
The symbol $[I_\lambda]$ denotes the class of the fixed point $I_\lambda$ 
in the Grothendieck group.
Below we will often omit the bracket $[ \ ]$ for simplicity.

Let us close this subsection with the remark on the operator $\nabla$. 
Consider the determinant line bundle $\cal{L} := \wedge^n \cal{B}$ of 
the tautological bundle $\cal{B}$ on $X^{[n]}$.
By \eqref{eq:chB} the $\bb{T}$-character of the fiber of $\cal{L}$ 
at the fixed point $I_\lambda$ is given by 
$$
 \ch \cal{L}_{I_\lambda} = q^{n(\lambda')}t^{n(\lambda)} = e_n[B_\lambda(q,t^{-1})],
$$
which is the eigenvalue of the operator $\nabla$ on $J_\lambda(q,t^{-1})$.
It implies \cite[Proposition 5.4.9]{H:2003} that 
\begin{align}
\label{eq:nabla-L}
 \Phi( - \otimes  \cal{L}) = \nabla \Phi(-).
\end{align}

More generally we can interpret the operator $\Delta_{f}$ 
\eqref{eq:Delta_f} on $\Lambda_{\bb{F}}$ in terms of geometry.
Let us recall the Schur functor.
For a partition $\lambda$ of $n$,
the Schur functor $S^{\lambda}$ on the category of $S_n$-modules 
is defined by
$$
 S^{\lambda}W := \Hom_{S_n}(V^{\lambda},W^{\otimes n}),
$$
where $V^{\lambda}$ is the irreducible representation of $S_n$ 
corresponding to the partition $\lambda$.
Now Haiman showed \cite[Proposition 5.4.9]{H:2003} that 
\begin{align}\label{eq:Delta-SB}
 \Phi( - \otimes  S^{\lambda}\cal{B}) = \Delta_{s_\lambda} \Phi(-).
\end{align}
Here $s_\lambda \in \Lambda^n$ is the Schur function.
The case $\lambda=(1^n)$ corresponds to the operator $\nabla$ since $s_{(1^n)} = e_n$.

\subsection{Geometric meaning of the Whittaker vector}
\label{ssec:geom:whit}

Let us rewrite our result in terms of the geometry of the Hilbert scheme $X^{[n]}$.
The definition \eqref{eq:HaimanH} of the modified Macdonald polynomial 
implies  that we should replace $t$ with $t^{-1}$ and 
apply the plethystic transformation $p_r \mapsto p_r/(1-t^{-r})$ 
to our result.
The transformed formula will be denoted by the bold symbol.

We start with the series $v^{(0)}(w)$.
\begin{align*}
\mathbf{v}^{(0)}(w) 
 = \exp\Bigl(\sum_{n>0} \dfrac{1}{(1-q^{n})(1-t^{n})}\dfrac{p_n}{n} w^{n}\Bigr).
\end{align*}
In terms of the modified Macdonald polynomials, we have 
$$
 \mathbf{v}^{(0)}(w) 
 = \sum_{\lambda} w^{|\lambda|} 
   \dfrac{\widetilde{H}_\lambda(q,t)}{c_\lambda(q^{-1},t) c'_\lambda(q,t^{-1})},
$$
which is obtained by specializing Cauchy kernel formula as in Lemma \eqref{eq:v0}.
In terms of the geometry, this expansion has the following meaning.

\begin{lem}
\begin{align}\label{eq:v0:geom}
\mathbf{v}^{(0)}(w) = \sum_{n\ge0} w^n \Phi(\cal{O}_{X^{[n]}}).
\end{align}
\end{lem}

\begin{proof}
In fact, by the Atiyah-Bott-Lefschetz localization formula, we have
\begin{align}\label{eq:ABL}
 \iota_*^{-1}[-] = 
 \sum_{\lambda \,\vdash n} \dfrac{\iota_{\lambda}^*[-]}{\wedge_{-1}T_{I_\lambda}^*X^{[n]}},
\end{align}
Here some explanations are in order.
$\iota: (X^{[n]})^{\bb{T}} \hookrightarrow X^{[n]}$ is the inclusion 
of the $\bb{T}$-fixed points 
and $\iota_{\lambda}: \{I_\lambda\} \hookrightarrow X^{[n]}$ is the inclusion 
of each $\bb{T}$-fixed point $I_\lambda$.
The push-forward 
$\iota_* : K_{\bb{T}}((X^{[n]})^{\bb{T}}) \longrightarrow K_{\bb{T}}(X^{[n]})$
induces isomorphism after localization due to Thomason's result.
Thus in \eqref{eq:ABL} $\iota_*^{-1}$ is well-defined.
(For more accounts, see \cite[\S4]{NY}.)

Then one can show the formula \eqref{eq:v0:geom}
by applying $\Phi$ on both sides of \eqref{eq:ABL} and using \eqref{eq:ch:T}.
\end{proof}

Next we turn to 
$$
 v_n(u) = \sum_{r\ge0}u^r \Delta_{h_r}v_n^{(0)}.
$$
Let us apply Haiman's result \eqref{eq:Delta-SB} 
in the case  $\lambda = (n)$.
Then since $s_{(n)} = h_n$ and by the previous Lemma,
we immediately have

\begin{prop}
\begin{align*}
 \mathbf{v}_n(u) 
 = \sum_{r\ge0} u^r \Delta_{h_r} \mathbf{v}_n^{(0)}
 = \sum_{r\ge0} u^r \Phi(S^n\cal{B}).
\end{align*}
\end{prop}

\begin{rmk}
We can also interpret $v_n^{(\infty)}$ \eqref{eq:vinf} using the determinant line bundle $\cal{L}$.
Recall that we have $v_n^{(\infty)} = \nabla^{-1} v_n^{(0)}$ by \eqref{eq:vinf-v0}.
Recall also Haiman's result \eqref{eq:nabla-L}. Thus, we have
$$
 \mathbf{v}^{(\infty)}(w) = \nabla^{-1} \mathbf{v}^{(\infty)}(w) 
 = \sum_{n\ge0} w^n \Phi(\cal{L}^*).
$$
Here $\cal{L}^*$ dentos the dual bundle of $\cal{L}$.
\end{rmk}

\subsection{Geometric interpretation of $\eta_1$ }
\label{ssec:goem:uqt}

Since we have succeeded in the description of $v_{G,n}$ 
via the geometry of Hilbert schemes of points,
it is natural to seek an interpretation of the operators $T_d$, 
$\eta_d$, or $D_d$ in terms of geometry.
However, there seem no simple interpretation of those operators for general $d$.
In the case $d=1$, a simple geometric description for $\eta_1$ is already known 
by \cite{FT,SV}.
Let us briefly give some comments.

In the previous subsections we considered the equivariant $K$-group 
$K_{\bb{T}}(X^{[n]})$ with fixed $n$.
If we want to interpret operators like $T_d$,
they should be correspond to the operators 
$K_{\bb{T}}(X^{[n]}) \to K_{\bb{T}}(X^{[n-d]})$ for any $n$.
Thus it is natural to consider the direct sum 
$\oplus_{n\ge0} K_{\bb{T}}(X^{[n]})$, 
which is identified with the whole Fock space $\cal{F}$ via $\Phi$.

The idea of geometric description of operators acting on the Heisenberg Fock space 
goes back to Nakajima's work \cite{N:1997,N:1999,N:2014},
which realized the Heisenberg Fock space on the equivariant homology groups 
$\oplus_{n\ge0} H^{\bb{T}}_{*}(X^{[n]})$.
The generating operators of Heisenberg algebra were realized by Hecke correspondences.
These correspondences are associated to the nested Hilbert schemes $X^{[n,n + d]}$.
For $d \ge 0$,
$X^{[n,n + d]}$ is a reduced closed subscheme of $X^{[n]} \times X^{[n+d]}$
parameterizing pairs $(I,J)$ of ideals 
in $\bb{C}[x,y]$ such that $J \subset I$.
The definition for $d \le 0$ is similar.
Then the associated Hecke correspondence yields an operator
$H^{\bb{T}}_{*}(X^{[n]}) \to H^{\bb{T}}_{*}(X^{[n-d]})$.

The works \cite{FT,SV} are natural $K$-theoretic analogue of Nakajima's work.
Let us describe the $K$-theoretic Heck correspondences in detail.
For smooth quasi-projective varieties $Y_1,Y_2,Y_3$ 
with action of a complex algebraic group $G$,
Assume that $Y_i$'s are proper over another quasi-projective $G$-variety $Y$. 
Then the convolution product is the map 
\begin{align*}
\begin{array}{ccc}
  K_{G}(Y_1 \times_Y Y_2) \otimes_{R(G)} 
  K_{G}(Y_2 \times_Y Y_3) 
& \longrightarrow
& K_{G}(Y_1 \times_Y Y_2)  \\
 ([c_1],[c_2])
& \longmapsto
& [\mathbf{R}p_{13*}(p_{12}^*(c_1) \otimes^{\mathbf{L}} p_{23}^*(c_2)].
\end{array}
\end{align*}
Here $p_{ij}$ is the projection.
We want to consider the case $Y_i = X^{[n_i]}$, 
$Y = \{\mathrm{pt}\}$ and $G = \bb{T}$,
although it is not possible since $X^{[n]}$ is not proper.
Considering the localized $K$-groups instead, we have the well-defined product
\begin{align*}
 K_{\bb{T}}(X^{[n_1]} \times X^{[n_2]})_{\loc} \otimes_{R(\bb{T})} 
 K_{\bb{T}}(X^{[n_2]} \times X^{[n_3]})_{\loc} 
 \longrightarrow
 K_{\bb{T}}(X^{[n_1]} \times X^{[n_3]})_{\loc}. 
\end{align*}
Here we defined 
$ K_{\bb{T}}(-)_{\loc} := K_{\bb{T}}(-)\otimes_{R(\bb{T})} \mathrm{Frac}(R(\bb{T}))
 \cong K_{\bb{T}}(-)\otimes_{R(\bb{T})} \bb{F}$. 
Under this product,
the $\bb{F}$-vector space 
$E := \oplus_{k \in \bb{Z}} \prod_{n} K_{\bb{T}}(X^{[n+k]} \times X^{[n]})_{\loc}$
has an associative algebra structure.
The same formula with $n_3=0$ implies that 
the $\bb{F}$-vector space 
$V := \oplus_{n \ge 0} K_{\bb{T}}(X^{[n]})$
is an $E$-module.

Consider the tautological bundles $\cal{B}_{n+1,n}$, $\cal{B}_{n,n}$ 
and $\cal{B}_{n,n+1}$ of the nested Hilbert schemes  
$X^{[n+1,n]}$, $X^{[n,n]}$ and $X^{[n,n + 1]}$.
Set 
$$
 u_{-1,l} := -q^{1/2} \prod_{n\ge0} [\cal{B}_{n,n+1}^{\otimes l}],\qquad 
 u_{ 1,l} :=  t^{1/2} \prod_{n\ge0} [\cal{B}_{n+1,n}^{\otimes l-1}]
$$
for $l\in\bb{Z}$ and 
$$
 u_{0, k} := \prod_{n\ge0} \psi_k([\cal{B}_{n,n}])-\dfrac{1}{(1-q^k)(1-t^k)},\qquad
 u_{0,-k} := -\prod_{n\ge0} \psi_k([\cal{B}_{n,n}^*])+\dfrac{1}{(1-q^{-k})(1-t^{-k})}
$$
for $k \in \bb{Z}_{\ge 1}$.
$\psi_k$ denotes the $k$-th Adams operation \cite{A} of $K$-theory. 
Then one can consider the subalgebra $\cal{U}$ of $E$ generated by these elements.

The main result of \cite{FT,SV} is that $\cal{U}$ is isomorphic to 
the Ding-Iohara-Miki algebra $\cal{U}(q,t^{-1})$,
and $V$ is isomorphic to the Fock representation.
The generators $u_{\pm1,d}$ and $u_{0,d}$ coincide 
with $u_{r,d}$ in the Hall algebra description explained at \eqref{eq:uqt:SV}.

As an immediate consequence, $\eta_1 = u_{-1,0}$ corresponds to 
(the $K$-group class of) the structure sheaf $\prod_{n} [\cal{O}_{X^{[n+1,n]}}]$.
 
A technical point is that the generating elements are constructed from 
the nested Hilbert scheme $X^{[n,n + d]}$ with $d=0,\pm1$.
Therefore, although we can construct operators  with $|d|>1$ 
as a convolution product of operators of $|d|=1$,
it is not always possible to obtain a simple formula.

%
%
%

\appendix
\section{Combinatorial properties of $T_d v_{G,n}$}

In the main text we proved the Whittaker condition of $v_{G,n}$ 
by computing formulas such as $\eta_d v_{G,n}$ and $D_d v_{G,n}$.
In this appendix we study properties of $T_d v_{G,n}$ directly.

\subsection{Duality of Macdonald symmetric functions and its consequence}
\label{ssec:app:dual}

Macdonald symmetric functions enjoy beautiful properties of dualities,
which was observed in \cite[Chap. VI, \S 5]{M}.
Recall the automorphism
$$
 \omega_{q,t}: p_n \longmapsto (-1)^{n-1}\dfrac{1-q^n}{1-t^n}p_n
$$
on the space $\Lambda_{\bb{F}}$ of symmetric functions,
which was introduced in \cite[Chap. VI, (2.14)]{M}.
Its inverse is given by
$$
 \omega_{t,q}: p_n \longmapsto (-1)^{n-1}\dfrac{1-t^n}{1-q^n}p_n.
$$
As proved in \cite[Chap. VI, \S 5]{M},
the automorphism $\omega_{q,t}$ acts on $P_\lambda(q,t)$ as
\begin{equation}
\label{eq:invol:P}
 \omega_{q,t} P_\lambda[X;q,t] = Q_{\lambda'}[X;t,q]. 
\end{equation}
Let us also recall the following properties
(see  \cite[Chap. VI, \S 4]{M}).
\begin{equation}
\label{eq:inverse:P}
 P_\lambda[X;1/q,1/t] = P_{\lambda}[X;q,t],\quad
 Q_\lambda[X;1/q,1/t] = (q/t)^{|\lambda|} Q_{\lambda}[X;q,t].
\end{equation}

Now we introduce another automorphism $\iota_{q,t}$ on  $\Lambda_{\bb{F}}$.
$$
 \iota_{q,t}: f[X;q,t] \longmapsto f[X;1/t,1/q].
$$
Following the idea and the notation in \cite{BGHT}, 
we define the operator 
\begin{align*}
 \downarrow := \iota_{q,t}\omega_{q,t}.
\end{align*}
Using the plethystic notation,
the action of $\downarrow$ on a symmetric function $f[X;q,t]$ can be written as 
\begin{align*}
 \downarrow f[X;q,t] = 
 f[-\ep X \tfrac{1-1/t}{1-1/q}; 1/t,1/q]
\end{align*}
Here $\ep$ is the plethystic minus symbol \cite{BGHT,H:2003,N:2012}.
(Recall that we have $p_n[\ep X] = (-1)^n p_n[X]$.)
By \eqref{eq:invol:P} and \eqref{eq:inverse:P} we have
\begin{align}
\label{eq:Jlow}
 \downarrow J_\lambda[X;q,t] 
 = (-t)^{|\lambda|} q^{n(\lambda')} t^{n(\lambda)} 
   J_{\lambda'}[X;q,t].
\end{align}

Define $v_{G,\lambda}$ by the formula 

\begin{align*}
 v_{G,\lambda} :=
 P_{\lambda}(q,t) \gamma_\lambda,
\end{align*}
where $\gamma_\lambda$ is given in \eqref{eq:AY}.
We have $v_{G,n} = \sum_{\lambda\,\vdash n} v_{G,\lambda}$.

\begin{lem}
\label{lem:vlow}
We have
\begin{align*}
\downarrow v_{G,\lambda} = (-q)^{|\lambda|} v_{G,\lambda'}
\end{align*}
for any partition $\lambda$.
Therefore we also have 
\begin{align*}
\downarrow v_{G,n} = (-q)^{n} v_{G,n}
\end{align*}
for any $n \in \bb{Z}_{\ge 0}$.
\end{lem}
\begin{proof}
By simple calculations using \eqref{eq:AY:J} and \eqref{eq:Jlow}.
Note that the parameter $k q/t$ is unchanged under the operation $\downarrow$.
\end{proof}

Let us recall the deformed Virasoro operators in the bosonized form.
\begin{align*}
&T(z) = \sum_{d \in \bb{Z}} T_d z^{-d},\quad
 T_d  = (q/t)^{d/2+1}  k      \Lambda^{+}_{d} 
      + (q/t)^{-d/2-1} k^{-1} \Lambda^{-}_{d} , 
\\
& \Lambda^{\pm}(z) = \sum_{d \in \bb{Z}} \Lambda^{\pm}_d z^{-d}
  = \exp\Bigl(
         \pm \sum_{n=1}^{\infty}
            \dfrac{1-t^{-n}}{1+(q/t)^n} \dfrac{p_{n}[X]}{n} z^{n}
         \Bigr)
     \exp\Bigl(
         \mp \sum_{n=1}^{\infty}
            (1-q^n) \partial_{p_n[X]} z^{-n}
         \Bigr).
\end{align*}

We have the following symmetry between $\Lambda^{\pm}(z)$.
\begin{lem}
\label{lem:Llow}
For an element $f[X;q,t] \in \Lambda_{\bb{F}}$ we have
\begin{align*}
 \downarrow \Lambda^{\pm}(z) \downarrow f[X;q,t] = \Lambda^{\pm}(\ep t z) f[(q/t)X;q,t]. 
\end{align*}
\end{lem}

\begin{proof}
Following \cite{BGHT}, we introduce the plethystic symbol
\begin{align*}
 \Omega[X] := \exp\Bigl(\sum_{m>0} \dfrac{p_m[X]}{m}\Bigr).
\end{align*}
Then one can write the action of the operator $\Lambda^{\pm}(z)$ as
\begin{align*}
 \Lambda^{\pm}(z) f[X;q,t] = 
 \Omega\bigl[\pm \tfrac{1-1/t}{1+q/t} z X\bigr]
 f[X \mp \tfrac{1-q}{z};q,t]
\end{align*}
Then we have
\begin{align*}
  \downarrow \Lambda^{\pm}(z) \downarrow f[X;q,t]
&=\downarrow \Lambda^{\pm}(z) 
   f\bigl[ -\ep \tfrac{1-1/t}{1-1/q} X ;1/t,1/q\bigr]
\\
&=\downarrow 
    \Omega\bigl[\pm \tfrac{1-1/t}{1+q/t} z X\bigr]
    f\bigl[ -\ep \tfrac{1-1/t}{1-1/q} (X \mp \tfrac{1-q}{z}) ;1/t,1/q \bigr]
\\
&=\iota_{q,t}
   \Omega\bigl[ \mp \ep \tfrac{1-q}{1-t} \tfrac{1-1/t}{1+q/t} z X \bigr]
   f\bigl[ -\ep \tfrac{1-1/t}{1-1/q} (-\ep \tfrac{1-q}{1-t} X \mp \tfrac{1-q}{z}) ;1/t,1/q \bigr]
\\
&= \Omega\bigl[ \mp \ep \tfrac{1-1/t}{1-1/q} \tfrac{1-q}{1+q/t} z X \bigr]
   f\bigl[ -\ep \tfrac{1-q}{1-t} (-\ep \tfrac{1-1/t}{1-1/q} X \mp \tfrac{1-1/t}{z}) ;q,t \bigr]
\\
&= \Omega\bigl[\pm q \ep \tfrac{1-1/t}{1+q/t} z X \bigr]
   f\bigl[(q/t) X \mp (\ep /t) \tfrac{1-q}{z} ;q,t \bigr].
\end{align*}
Thus the consequence holds.
\end{proof}

By a simple calculation using Lemma \ref{lem:vlow} and Lemma \ref{lem:Llow},
we have

\begin{prop}
For each $d,n \in \bb{Z}_{\ge 0}$ we have 
\begin{align*}
 \downarrow \Lambda^{\pm}_d v_n = (-q)^{n-d} \Lambda^{\pm}_d v_n. 
\end{align*}
\end{prop}

\subsection{Combinatorial formula for $T_{|\lambda|} J_{\lambda}(q,t)$}
\label{ssec:app:TnJn}

Unfortunately we don't have a nice method to calculate $T_d v_{G,n}$ directly.
However we can do some calculation in the case $d=n$.
To explain it, let us write down an explicit formula for 
$\Lambda^{\pm}_{|\lambda|} J_\lambda(q,t)$.

To compute $\Lambda^{+}_{|\lambda|}J_\lambda(q,t)$,
it is enough to calculate the coefficient of $z^{-|\lambda|}$ in 
\begin{align}\label{eq:app:L+J}
 \exp\Bigl(-\sum_{m>0}(1-q^m)\partial_{p_m} z^{-m}\Bigr)
 J_\lambda(q,t).
\end{align}
By the Macdonald inner product $\left<\cdot,\cdot\right>_{q,t}$
we have
\begin{align}\label{eq:app:ip}
\Bigl<1,
     \exp\Bigl(-\sum_{m>0}(1-q^m)\partial_{p_m} z^{-m}\Bigr)
     J_\lambda(q,t)
\Bigr>_{q,t}
=
\Bigl<\exp\Bigl(-\sum_{m>0}(1-t^m)\dfrac{p_m}{m} z^{-m}\Bigr),
      J_\lambda(q,t)
\Bigr>_{q,t}.
\end{align}
To write down the result, let us define $\tau_\lambda \in \bb{F}$ by 
$$
 \exp\Bigl(-\sum_{m>0}(1-t^m) \dfrac{p_m}{m} z^{-m}\Bigr)
 = \sum_{\lambda} z^{-|\lambda|} \tau_\lambda P_\lambda(q,t).
$$

\begin{lem}\label{lem:app:tau}
$\tau_\lambda$ is given by 
\begin{align*}
 \tau_\lambda = 
 (1-q)(t-1) t^{|\lambda|-1+n(\lambda)} e_1[B_{\lambda}] 
   \dfrac{ \prod_{s \in \lambda \setminus {(1,1)}} (1 - B_s)}
         {c'_\lambda(q,t)}.
\end{align*}
Here we used the symbol $B_\lambda=\{B_s\mid s\in \lambda\}$ defined at 
\eqref{eq:B} and \eqref{eq:Bs}.
\end{lem}

\begin{proof}
We just give a sketch of proof.
Using the generating functions of 
elementary and complete symmetric functions $e_n$'s and $h_n$'s,
we have
\begin{align*}
 \exp\Bigl(-\sum_{n>0}\dfrac{1-t^n}{n}p_n z^n\Bigr)
&
=\exp\Bigl(-\sum_{n>0}\dfrac{p_n}{n} z^n\Bigr)
 \exp\Bigl( \sum_{n>0}\dfrac{p_n}{n} (t z)^n\Bigr)
=\Bigl(\sum_{l\ge0}(-1)^l e_l z^l \Bigr)
 \Bigl(\sum_{m\ge0} h_m t^m z^m \Bigr)
\\
&
=\sum_{n\ge0}z^n 
 \Bigl(\sum_{m=0}^{n} (-1)^{n-m} e_{n-m} h_m t^m \Bigr).
\end{align*}
Thus we have
\begin{align*}
 \sum_{\lambda \vdash n} \tau_\lambda J_\lambda(q,t) 
=\sum_{m=0}^n (-1)^{n-m} e_{n-m} h_m t^m. 
\end{align*}

Then setting 
$$
 R_n := \sum_{m=0}^n (-1)^{n-m} e_{n-m} h_m t^m
$$ 
we have
\begin{align}
\label{eq:R_n}
 R_n = 
 \sum_{r=1}^{n-1} (-1)^{r-1} t^r e_r R_{n-r} 
 + (-1)^{n-1} (t^n-1) e_n. 
\end{align}
This relation follows from the formula 
$\sum_{r=0}^n (-1)^r e_r h_{n-r}=0$
and a direct calculation.

Using this relation \eqref{eq:R_n} and the Pieri formula 
\cite[Chap VI, \S 6, (6.24)]{M}
$P_\mu(q,t) e_r = \sum_{\lambda} \psi'_{\lambda/\mu} P_\lambda$,
we can show the following recursive relation for $\tau_\lambda$.
\begin{align}
\label{eq:tau:rec}
 \tau_\lambda 
= \sum_{r=1}^{n-1}(-1)^{r-1} t^r
   \sum_{\mu \vdash |\lambda|-r} \tau_\mu \psi'_{\lambda/\mu}
  +\delta_{\lambda,(1^n)} (-1)^{n-1} (t^n-1).
\end{align}

Now the statement of Lemma \ref{lem:app:tau} can be shown inductively by the initial formula 
$\tau_{(1)}=t-1$ and the recursive formula \eqref{eq:tau:rec}.
In fact, for $\lambda = (n)$, \eqref{eq:tau:rec} reads
$\tau_{(n)} = t \tau_{(n-1)}\psi'_{(n)/(n-1)}$.
Thus we immediately have $\tau_{(n)} = t^{n-1}(t-1)$.
One can perform the induction on the length of $\lambda$.
We omit the detail.
\end{proof}

By \eqref{eq:app:ip},
the coefficient of $z^{-|\lambda|}$ in \eqref{eq:app:L+J}
is given by 
$\tau_\lambda \left<P_\lambda,J_\lambda\right>_{q,t}
 =\tau_\lambda c'_\lambda(q,t)$.
Then by Lemma \ref{lem:app:tau} we have

\begin{prop}
\begin{align*}
 \Lambda^{+}_{|\lambda|} J_{\lambda}(q,t)
 = k (q/t)^{|\lambda|/2}
  (1-q)(t-1) t^{|\lambda|-1+n(\lambda)} e_1[B_{\lambda}] 
  \prod_{s \in \lambda \setminus {(1,1)}} (1-B_s)
\end{align*}
\end{prop}

Similarly we have 
$$
 \exp\Bigl(\sum_{m>0}(1-t^m)\dfrac{p_m}{m} p_m z^{-m}\Bigr)
 = \sum_{\lambda} z^{-|\lambda|} \widetilde{\tau}_\lambda P_\lambda(q,t) 
$$
with 
\begin{align*}
 \widetilde{\tau}_\lambda = 
 (1-q)(1-t) q^{|\lambda|-1}t^{n(\lambda)} e_{-1}[B_\lambda]
   \dfrac{ \prod_{s \in \lambda \setminus {(1,1)}} (1 -  B_s)}
         { c'_\lambda(q,t)}.
\end{align*}
Here we used the symbol 
$e_{-1}[X] := X_1^{-1} + X_2^{-1} + \cdots$.

Therefore we have 

\begin{prop}
\begin{align*}
 \Lambda^{-}_{|\lambda|} J_{\lambda}(q,t)
 = k^{-1} (q/t)^{-|\lambda|/2}
 (1-q)(1-t) q^{|\lambda|-1}t^{n(\lambda)} e_{-1}[B_\lambda]
  \prod_{s \in \lambda \setminus {(1,1)}} (1 -  B_s)
\end{align*}
\end{prop}

Using these Propositions, we can give another proof 
of the Whittaker condition $T_n v_{G,n} = 0$.
By the relation \eqref{eq:qvir:bosonization} of $T(z)$ and $\Lambda^{\pm}(z)$, 
the definition \eqref{eq:AY} of $v_{G,n}$,
the condition is equivalent to the equation
\begin{align*}
 \sum_{\lambda \, \vdash n}
 \dfrac{ q^{n(\lambda')} t^{n(\lambda)} \prod_{s \in \lambda \setminus {(1,1)}} (1 -  B_s)}
       {c_\lambda(q,t) c'_\lambda(q,t)}
 \dfrac{u e_1[B_\lambda] - e_{-1}[B_\lambda]}{\prod_{s \in \lambda } (1 - u B_s)} = 0.
\end{align*}
Here $u$ is an arbitrary indeterminate.
This formula can be shown by a specialization of Cauchy kernel.
We omit the detail.

%
%


\end{document}